\documentclass[preprint,12pt]{elsarticle}

\usepackage{amsmath,amssymb,amsfonts,amsthm,mathrsfs}
\usepackage{geometry}
\usepackage{hyperref}
\usepackage{booktabs}
\usepackage{array}
\usepackage{tabularx}
\usepackage{enumitem}
\usepackage{mathtools}

\hypersetup{
  colorlinks=false,
  hidelinks
}

\theoremstyle{plain}
\newtheorem{theorem}{Theorem}[section]
\newtheorem{lemma}[theorem]{Lemma}
\newtheorem{proposition}[theorem]{Proposition}
\newtheorem{corollary}[theorem]{Corollary}

\theoremstyle{definition}

\theoremstyle{remark}
\newtheorem{remark}[theorem]{Remark}
\newtheorem{example}[theorem]{Example}

\newcommand{\R}{\mathbb{R}}
\newcommand{\Z}{\mathbb{Z}}
\newcommand{\cP}{\mathcal{P}}
\newcommand{\Ehr}{\mathrm{Ehr}}
\newcommand{\vol}{\mathrm{vol}}

\title{Alternating adjacent-sum polytopes: transfer matrices and Ehrhart series}

\author[hn]{Xinru Jiang}
\ead{chiang\_11110@163.com}

\author[csu]{Suzhen Wen}
\ead{wensuzhencs@csu.edu.cn}

\author[hn]{Yueming Zhong\corref{cor1}}
\ead{zhongyueming107@gmail.com}
\cortext[cor1]{Corresponding author.}

\address[hn]{School of Mathematics and Statistics, Hainan University, Haikou 570228, P.R. China}
\address[csu]{School of Mathematics and Statistics, HNP-LAMA, Central South University, Changsha 410083, P.R. China}

\begin{document}

\begin{frontmatter}

\begin{abstract}
We study a period-two family of adjacent-sum lattice polytopes, defined by
consecutive-coordinate inequalities whose bounds alternate between $s$ and
$s{+}1$.  This family is a particularly simple non-uniform deformation of the classical
uniform adjacent-sum model that still admits an explicit transfer-matrix
analysis.

A first structural difference from the uniform model is a \emph{parity
split}: the lattice-point counts in odd and even dimensions satisfy two
genuinely different rational generating functions with a common denominator.
The odd-dimensional generating function satisfies a M\"{o}bius recurrence
and admits a closed form in terms of the arctangent; the even-dimensional
series satisfies a coupled recurrence.  Both share an explicit dominant pole
that governs the exponential growth rate in each parity class.

As a companion result, closing the path by adding a cyclic constraint on the
first and last coordinates turns the lattice-point count into a matrix trace.
Both cyclic parity classes admit rational generating functions with the same
denominator as the open-path series.  The even-dimensional cyclic series has
a Jacobi-derivative form, while the odd-dimensional one is expressed through
an explicit anti-diagonal cofactor numerator.

On the Ehrhart side, we compute dimension-generating functions for fixed
dilations, derive linear recurrences for the lattice-point counts, obtain
rational volume-generating functions, and establish a bivariate identity
for the coefficients of the $h^*$-polynomials.

For the parameter value $s=1$, the even-dimensional polytopes decompose as
Cartesian products of two-dimensional unimodular simplices, yielding explicit
counting formulas and the Gorenstein property in every even dimension.  For
every $s\ge2$ the Gorenstein property fails in some even dimension; the failure
occurs in dimension two for all $s\ne3$, and in dimension four for $s=3$.
\end{abstract}


\end{frontmatter}
\thispagestyle{empty}

\section{Introduction}\label{sec-intro}

The enumeration of integer points in lattice polytopes is a central theme in
algebraic and enumerative combinatorics.  If $\cP\subset\R^d$ is a lattice
polytope, Ehrhart's theorem \cite{Ehrhart1,Ehrhart} asserts that
\[
L_{\cP}(m):=\#(m\cP\cap\Z^d), \qquad m\in\Z_{\ge0},
\]
is a polynomial in $m$ of degree $d$.  Its generating function has the form
\[
\Ehr_{\cP}(z)=1+\sum_{m\ge1}L_{\cP}(m)z^m
=\frac{h^*_{\cP}(z)}{(1-z)^{d+1}}.
\]
Following the work of Stanley, Hibi, and others, the numerator
$h^*_{\cP}(z)$ has become a fundamental invariant of a lattice polytope: its
coefficients are nonnegative integers \cite{Stanley1978}, their sum is the
normalized volume, and palindromicity is equivalent to the Gorenstein property
of the Ehrhart ring \cite{hibi-1,Stanley1978}; see also \cite{Bruns}.  Problems
concerning unimodality, log-concavity, and real-rootedness of $h^*$-polynomials
remain active in Ehrhart theory \cite{beck,BrandenSolus,Braun}.

This paper concerns lattice polytopes defined by local constraints.  Such
families occur naturally in the study of order and chain polytopes and their
variants \cite{ArdilaBliem,OhsugiTsuchiyaEnriched,Ramassamy,StanleyTwo}.  From
a graph-theoretic viewpoint, inequalities of the form
\[
x_i+x_{i+1}\le c_i
\]
are capacitated stable-set-type constraints on a path.  The homogeneous case
$c_i=s$ is a one-step finite-state system, while the period-two choice
$c_i=s,s+1,s,s+1,\ldots$ is the basic non-homogeneous case whose
non-constant local transfer matrix still admits an explicit analysis via a
two-step product structure.
A particularly tractable model is the uniform adjacent-sum family
\[
x_i+x_{i+1}\le s\quad(1\le i\le d-1),\qquad x_j\ge0,
\]
which has been studied from several angles: B\'{o}na--Ju \cite{Bona2} solved
the magic-square-type linear inequalities that first motivated the family;
B\'{o}na--Ju--Yoshida \cite{Bona} extended the enumeration to related
weighted-graph problems; and Xin--Zhong \cite{Xin} gave closed-form Kekul\'{e}
number identities via Chebyshev polynomials, using a transfer-matrix structure
closely related to the one exploited here.  Transfer matrices are ideally suited to locally
constrained systems because each new coordinate appends one inequality,
translating into a matrix product; the resulting geometric-series formula for
the resolvent $(I-yC)^{-1}$ directly produces rational generating functions,
and Engstr\"{o}m--Kohl \cite{EngstromKohl} showed formally how to translate
this resolvent-based counting into Ehrhart series and $h^*$-data.

We investigate the following period-two non-uniform deformation of the uniform
adjacent-sum model: the upper bounds are allowed to alternate with period two.
For integers $d\ge2$ and $s\ge1$, define
\begin{equation}\label{eq-poly-def}
\cP_d^{(s)}
=\Bigl\{
x=(x_1,\dots,x_d)\in\R_{\ge0}^d:
 x_i+x_{i+1}\le s+\delta_i,\quad 1\le i\le d-1
\Bigr\},
\end{equation}
where $\delta_i=0$ for $i$ odd and $\delta_i=1$ for $i$ even.  Thus the bounds
are $s,s+1,s,s+1,\ldots$.  Since the constraint matrix is the edge--vertex
incidence matrix of a path graph, it is totally unimodular
\cite[Theorem~19.3]{Schrijver}; hence $\cP_d^{(s)}$ is a lattice polytope.

To see the alternation concretely, take $s=1$ and $d=4$: the constraints are
$x_1+x_2\le1$, $x_2+x_3\le2$, $x_3+x_4\le1$, with $x_i\ge0$.  Since
$x_2\le x_1+x_2\le1$ and $x_3\le x_3+x_4\le1$, the middle constraint
$x_2+x_3\le2$ is automatically satisfied; the lattice points of $\cP_4^{(1)}$
are the nine pairs
$((x_1,x_2),(x_3,x_4))\in\{(i,j):i+j\le1\}^2$, giving
$\#(\cP_4^{(1)}\cap\Z^4)=9$.
For comparison, the uniform model with \emph{all} bounds equal to $1$ has only
$N_4^{\mathrm{unif}}(1)=8$ lattice points in dimension $4$ (the relaxation of
even-indexed constraints from $1$ to $2$ admits one additional point).
The gap widens sharply in higher dimensions: the alternating model has
$N_5(1)=24$ lattice points in dimension $5$, versus $N_5^{\mathrm{unif}}(1)=13$
for the uniform model, foreshadowing the genuinely different generating
functions governing the two families.  This example illustrates the main
phenomenon of the paper: the period-two perturbation does not merely change
the numerical values, but separates the dimension counts by parity.

Our first observation is that this small perturbation already changes the
enumeration.  Write $N_d(s):=\#(\cP_d^{(s)}\cap\Z^d)$ for the lattice-point
count; the values $N_4(1)=9$ and $N_5(1)=24$ used above are special cases.
Unlike the uniform model, the odd- and even-dimensional counts are governed by
different series.  We use the auxiliary normalization
$\widetilde N_1(s):=s+2$ and define
\begin{equation}\label{eq-PhiPsi-def}
\Phi_s(y):=\widetilde N_1(s)+\sum_{m\ge1}N_{2m+1}(s)\,y^m,
\qquad
\Psi_s(y):=\sum_{m\ge0}N_{2m+2}(s)\,y^m.
\end{equation}
We also use the polynomial families
\begin{equation}\label{eq-PQ-def}
P_n(y):=\sum_{k=0}^{\lfloor(n-1)/2\rfloor}\binom{n}{2k+1}y^k,
\qquad
Q_n(y):=\sum_{k=0}^{\lfloor n/2\rfloor}\binom{n}{2k}y^k.
\end{equation}
They are the odd and even parts of $(1+\sqrt y)^n$, and after the substitution
$y=-\tan^2\theta$ they become rescaled Chebyshev polynomials.

\medskip
\noindent\textbf{Dimension generating functions.}
The first group of results gives closed forms for the dimension-varying
lattice-point counts.

\begin{theorem}[Rational dimension generating functions]\label{thm-rational}
For every integer $s\ge1$,
\begin{equation}\label{eq-Phi-main}
\Phi_s(y)=\frac{P_{s+2}(-y)}{Q_{s+2}(-y)},
\qquad
\Psi_s(y)=\frac{1-Q_{s+2}(-y)}{y\,Q_{s+2}(-y)}.
\end{equation}
The full dimension generating function $F_s(x):=\sum_{d\ge2}N_d(s)\,x^d$ satisfies
\begin{equation}\label{eq-F-main}
F_s(x)=x\bigl(\Phi_s(x^2)-(s+2)\bigr)+x^2\Psi_s(x^2).
\end{equation}
\end{theorem}

The same formulas imply the following recurrences.

\begin{theorem}[M\"{o}bius and coupled recurrences]\label{thm-recurrence}
For every integer $s\ge1$,
\begin{equation}\label{eq-Phi-rec-main}
\Phi_s(y)=\frac{1+\Phi_{s-1}(y)}{1-y\,\Phi_{s-1}(y)},
\qquad
\Phi_0(y)=\frac{2}{1-y},\quad\Phi_1(y)=\frac{3-y}{1-3y},
\end{equation}
\begin{equation}\label{eq-Psi-rec-main}
1+y\Psi_s(y)=\frac{1+y\Psi_{s-1}(y)}{1-y\Phi_{s-1}(y)},
\qquad\Psi_0(y)=\frac{1}{1-y}.
\end{equation}
In particular, $\Phi_s(y)$ admits an iterative continued-fraction expansion.
\end{theorem}

\begin{theorem}[Trigonometric closed forms]\label{thm-tangent}
For every integer $s\ge1$,
\begin{equation}\label{eq-Phi-tan-main}
\Phi_s(y)=\frac{\tan\!\bigl((s+2)\arctan\sqrt{y}\bigr)}{\sqrt{y}},
\end{equation}
\begin{equation}\label{eq-Psi-trig-main}
\Psi_s(y)=\frac{1}{y}\!\left(
\frac{\cos^{s+2}\!\bigl(\arctan\sqrt{y}\bigr)}%
     {\cos\!\bigl((s+2)\arctan\sqrt{y}\bigr)}-1\right).
\end{equation}
Moreover, set
\[
\alpha_s:=\frac{\pi}{2(s+2)},\qquad
 y^*:=\tan^2\alpha_s,
\]
and define
\begin{equation}\label{eq-kappa-lambda-def-intro}
\kappa_s(y^*) := \frac{2\sec^2\!\alpha_s}{(s+2)y^*},
\qquad
\lambda_s(y^*) := \frac{2\cos^{s+3}\!\alpha_s}{(s+2)\sin^3\!\alpha_s}.
\end{equation}
Then the two parity subsequences satisfy
\begin{equation}\label{eq-open-asymptotics-intro}
N_{2m+1}(s)\sim\kappa_s(y^*)\,(y^*)^{-m},
\qquad
N_{2m+2}(s)\sim\lambda_s(y^*)\,(y^*)^{-m}
\end{equation}
as $m\to\infty$.
\end{theorem}

\medskip
\noindent\textbf{Cyclic variants.}
Closing the path boundary by adding the constraint $x_d+x_1\le s+\delta_d$
(where $\delta_d=0$ for $d$ odd and $\delta_d=1$ for $d$ even) turns each
lattice-point count into the trace of a power of the double-step transfer
matrix $C_s$.  Define
\[
N_d^{\mathrm{cyc}}(s):=
\#\Bigl(\bigl\{x\in\cP_d^{(s)}:x_d+x_1\le s+\delta_d\bigr\}\cap\Z^d\Bigr).
\]
We use the two cyclic generating functions
\[
\Psi_s^{\mathrm{cyc}}(y):=\sum_{m\ge1}N_{2m}^{\mathrm{cyc}}(s)y^{m-1},
\qquad
\Omega_s^{\mathrm{cyc}}(y):=\sum_{m\ge1}N_{2m+1}^{\mathrm{cyc}}(s)y^{m-1}.
\]

\begin{theorem}[Cyclic generating functions]\label{thm-cyclic-intro}
For every integer $s\ge1$,
\begin{equation}\label{eq-Psicyc-intro}
\Psi_s^{\mathrm{cyc}}(y)=\frac{(s+2)P_{s+1}(-y)}{2Q_{s+2}(-y)}.
\end{equation}
The odd-dimensional cyclic series is less symmetric, but it is still rational with the same denominator.  More precisely,
\begin{equation}\label{eq-Omegacyc-intro}
\Omega_s^{\mathrm{cyc}}(y)=
\frac{\mathcal{R}_s(y)-\bigl(\lfloor s/2\rfloor+1\bigr)Q_{s+2}(-y)}{yQ_{s+2}(-y)},
\end{equation}
where $\mathcal{R}_s(y)$ is an explicit anti-diagonal cofactor polynomial,
given in Lemma~\ref{lem-Rs-explicit}.  Finally, let
\[
\alpha_s:=\frac{\pi}{2(s+2)},\qquad y_s^*:=\tan^2\alpha_s,
\qquad
\xi_s:=\frac{2\mathcal{R}_s(y_s^*)\sin\alpha_s\cos^{s-1}\alpha_s}{(s+2)y_s^*}.
\]
Then, as $m\to\infty$,
\begin{equation}\label{eq-cyclic-asymptotics-intro}
N_{2m}^{\mathrm{cyc}}(s)\sim 1\cdot (y_s^*)^{-m},
\qquad
N_{2m+1}^{\mathrm{cyc}}(s)\sim \xi_s (y_s^*)^{-m}.
\end{equation}
In particular, the leading constant in the even cyclic parity class is exactly $1$.
\end{theorem}

The denominator $Q_{s+2}(-y)$ is the same as in Theorems~\ref{thm-rational}
and~\ref{thm-tangent}.  Thus the cyclic closure preserves the same dominant
pole as the open-path family, but its numerators are governed by traces rather
than by endpoint sums.

The cyclic series also admit recurrence and trigonometric descriptions, parallel
to Theorems~\ref{thm-recurrence} and~\ref{thm-tangent}.

\begin{theorem}[Cyclic recurrences]\label{thm-cyclic-recurrence}
Set
\[
\widehat\Psi_0^{\mathrm{cyc}}(y):=\frac{1}{1-y},\qquad
\widehat\Psi_s^{\mathrm{cyc}}(y):=\frac{2}{s+2}\Psi_s^{\mathrm{cyc}}(y)\quad(s\ge1).
\]
Then the normalized even-dimensional cyclic series satisfies, for every $s\ge1$,
\begin{equation}\label{eq-Lcyc-rec-main}
\widehat\Psi_s^{\mathrm{cyc}}(y)=
\frac{1+(1+y)\widehat\Psi_{s-1}^{\mathrm{cyc}}(y)}
{1-y-y(1+y)\widehat\Psi_{s-1}^{\mathrm{cyc}}(y)}.
\end{equation}
Equivalently, the unnormalized series itself satisfies
\begin{equation}\label{eq-Psicyc-rec-main}
\Psi_s^{\mathrm{cyc}}(y)=
\frac{(s+2)\bigl((s+1)+2(1+y)\Psi_{s-1}^{\mathrm{cyc}}(y)\bigr)}
{2\bigl((s+1)(1-y)-2y(1+y)\Psi_{s-1}^{\mathrm{cyc}}(y)\bigr)},
\qquad
\Psi_0^{\mathrm{cyc}}(y)=\frac{1}{1-y}.
\end{equation}
The odd cyclic series has the extra numerator $\mathcal{R}_s(y)$ in
Theorem~\ref{thm-cyclic-intro}; rather than a scalar M\"obius recurrence in $s$,
it is governed by the same denominator recurrence.  Let
$\nu_s:=\lfloor(s+2)/2\rfloor$ and
\[
Q_{s+2}(-y)=\sum_{j=0}^{\nu_s}(-1)^j\binom{s+2}{2j}y^j.
\]
If a cyclic generating function has the form
\[
\sum_{m\ge1}a_m y^{m-1}=\frac{A(y)}{Q_{s+2}(-y)},
\]
where $A(y)$ is a polynomial, then
\begin{equation}\label{eq-cyclic-coeff-rec-main}
a_m=\sum_{j=1}^{\nu_s}(-1)^{j+1}\binom{s+2}{2j}a_{m-j}
\end{equation}
for all $m\ge \deg A+2$.  In particular, for the even cyclic sequence the
recurrence starts at $m\ge \nu_s+1$; for the odd cyclic sequence the same
recurrence holds after the finitely many initial terms determined by the numerator
of $\Omega_s^{\mathrm{cyc}}(y)$.
\end{theorem}

\begin{theorem}[Cyclic trigonometric forms]\label{thm-cyclic-trig}
Let $\theta=\arctan\sqrt y$ and
$\tau_s:=\lfloor s/2\rfloor+1$.  Then
\begin{equation}\label{eq-Psicyc-trig-main}
\Psi_s^{\mathrm{cyc}}(y)=
\frac{s+2}{2}\,
\frac{\sin((s+1)\theta)\cos^2\theta}
     {\sin\theta\,\cos((s+2)\theta)}.
\end{equation}
Moreover,
\begin{equation}\label{eq-Omegacyc-trig-main}
\Omega_s^{\mathrm{cyc}}(\tan^2\theta)=
\frac{\cos^{s+2}\theta\,\mathcal{R}_s(\tan^2\theta)-\tau_s\cos((s+2)\theta)}
     {\tan^2\theta\,\cos((s+2)\theta)},
\end{equation}
where $\mathcal{R}_s$ is the explicit polynomial given in Lemma~\ref{lem-Rs-explicit}.
\end{theorem}

Theorems~\ref{thm-rational}--\ref{thm-cyclic-trig} give the main
dimension-counting formulas for the open-path family and its cyclic companion.
Thus Section~\ref{sec-2} gives a complete transfer-matrix solution for both
open and cyclic boundary conditions.  We turn next from dimension counting to
the finer Ehrhart-theoretic structure of the individual polytopes $\cP_d^{(s)}$
for fixed $d$.

\medskip
\noindent\textbf{Ehrhart data from dimension generating functions.}
In the Ehrhart part, we use a dimension-generating approach.  Instead of
computing the Ehrhart series of each fixed-dimensional polytope separately, we
first compute fixed-dilation generating functions in the dimension variable and
then extract the $h^*$-coefficients.

For $d\ge2$, write
\[
L_d^{(s)}(m):=\#\bigl(m\cP_d^{(s)}\cap\Z^d\bigr),
\qquad
\Ehr_{\cP_d^{(s)}}(z)=\frac{h_d^{(s)}(z)}{(1-z)^{d+1}},
\]
and expand
\[
h_d^{(s)}(z)=\sum_{k=0}^{d}h_{d,k}^{(s)}z^k.
\]
Thus $h_{d,k}^{(s)}$ denotes the $k$-th coefficient of the Ehrhart
$h^*$-polynomial of $\cP_d^{(s)}$.  For the dimension-generating functions
below we also use the low-dimensional conventions
$L_0^{(s)}(m)=1$ and $L_1^{(s)}(m)=sm+1$, as specified again in
Section~\ref{subsec-ehrhart-setup}.  Now let
\[
G_m^{(s)}(y):=\sum_{d\ge0}L_d^{(s)}(m)y^d,
\qquad
H_k^{(s)}(y):=\sum_{d\ge0}h_{d,k}^{(s)}y^d.
\]
Here $\mathcal O_m^{(s)}$ and $\mathcal E_m^{(s)}$ will denote the odd- and
even-dimensional fixed-dilation series, respectively; $G_m^{(s)}$ combines the
two parity classes, and $H_k^{(s)}$ records the $k$-th $h^*$-coefficient across
all dimensions.
The following formula extracts the coefficient-generating functions of the
$h^*$-polynomials from the fixed-dilation series.

\begin{theorem}[General $h^*$-extraction]\label{thm-hstar-extraction}
For fixed $s\ge1$ and $k\ge0$, the generating function
$H_k^{(s)}(y)$ of the $k$-th $h^*$-coefficient satisfies
\begin{equation}\label{eq-Hk-extraction}
H_k^{(s)}(y)
=\sum_{i=0}^{k}\frac{(-1)^i}{i!}\,y^{i-1}
\frac{d^i}{dy^i}\!\Bigl(y\,G_{k-i}^{(s)}(y)\Bigr).
\end{equation}
\end{theorem}

We next give the explicit fixed-dilation formulas.  For $0\le j\le s$, define
\begin{equation}\label{eq-ajbj-def}
a_j^{(s,m)}:=\binom{(s+1-j)m+j+1}{2j},\qquad
b_j^{(s,m)}:=\binom{(s+1-j)m+j+1}{2j+1}.
\end{equation}
Note that $a_0^{(s,m)}=1$ for all $s,m$, since $\binom{(s+1)m+1}{0}=1$.
The combinatorial significance of $a_j^{(s,m)}$ will become clear in
Lemma~\ref{lem-Gm-triangular}: it is the inner product
$\mathbf c^\top U_j$ of the compression vector and the $j$-th test vector.
Define the denominator polynomial
\begin{equation}\label{eq-Deltasm-def}
\Delta_{s,m}(z):=\sum_{j=0}^{s}(-1)^j a_j^{(s,m)}z^j.
\end{equation}

\begin{theorem}[Fixed-dilation rational formulas]\label{thm-Gm-general}
For every $s,m\ge1$, the parity-split generating functions are
\begin{align}
\mathcal{E}_m^{(s)}(z)
&=\frac{\displaystyle\sum_{j=0}^{s-1}(-1)^j a_{j+1}^{(s,m)}z^j}{\Delta_{s,m}(z)},
\label{eq-Em-general}\\
\mathcal{O}_m^{(s)}(z)
&=\frac{\displaystyle\sum_{j=0}^{s}(-1)^j b_j^{(s,m)}z^j}{\Delta_{s,m}(z)}.
\label{eq-Om-general}
\end{align}
Moreover,
\begin{equation}\label{eq-Gms-general}
G_m^{(s)}(y)=1+y\bigl(\mathcal{O}_m^{(s)}(y^2)-m\bigr)+y^2\mathcal{E}_m^{(s)}(y^2).
\end{equation}
\end{theorem}

When $m=1$, these formulas specialize to the ordinary lattice-point counts of
$\cP_d^{(s)}$: indeed
\[
G_1^{(s)}(x)=1+(s+1)x+F_s(x).
\]
Thus Theorem~\ref{thm-Gm-general} recovers the rational formula for
$F_s(x)$ after the specialization $a_j^{(s,1)}=\binom{s+2}{2j}$ and
$b_j^{(s,1)}=\binom{s+2}{2j+1}$.  This specialization, however, does not by
itself explain the M\"{o}bius recurrence in $s$ or the trigonometric forms in
Theorems~\ref{thm-recurrence} and~\ref{thm-tangent}; those require the
Pascal/Chebyshev structure exposed by the direct transfer-matrix analysis in
Section~\ref{sec-2}.

These formulas imply recurrences of order at most $s$ for the fixed-dilation counts,
rational generating functions for the volumes, and a bivariate identity for
$\sum_{d\ge0}h_d^{(s)}(z)y^d$; see Section~\ref{subsec-consequences}.  They
also give explicit formulas for $H_k^{(s)}(y)$ for $k\le4$ in
Proposition~\ref{prop-Hk-low}.

\medskip
\noindent\textbf{Gorenstein behavior.}
The alternating family has a clear exceptional case.  When $s=1$, the
even-dimensional polytopes split as products of two-dimensional unimodular
simplices,
$\cP_{2r}^{(1)}\cong \mathcal T_1^r$, where
$\mathcal T_1=\{(a,b)\in\R_{\ge0}^2:a+b\le1\}$.  We obtain explicit lattice-point counts
and volumes and prove that $\cP_{2r}^{(1)}$ is Gorenstein for every $r\ge1$.

\begin{theorem}[Gorenstein property for $s=1$]\label{thm-gorenstein-s1}
For every integer $r\ge1$, the polytope $\cP_{2r}^{(1)}$ is Gorenstein,
equivalently, its $h^*$-polynomial is palindromic.
\end{theorem}

For $s\ge2$ the situation is different: for every $s\ge2$ there exists an even
dimension $d$ such that $\cP_d^{(s)}$ is not Gorenstein.  In fact the failure
already occurs in dimension $2$ unless $s=3$, and for $s=3$ it occurs in
dimension $4$.  This contrasts with the related uniform graph-polytope setting,
where the corresponding Ehrhart numerators are palindromic \cite{Xin}.

\medskip
\noindent\textbf{Comparison with the uniform model.}
Table~\ref{tab-comparison} highlights the structural differences between the
two families.  The uniform-family entries are taken from the open-path and
circular enumerations of B\'{o}na--Ju~\cite{Bona2}, from the graph-polytope,
continued-fraction, Chebyshev and spectral formulas of Xin--Zhong~\cite{Xin},
and from the zig-zag chain-polytope unimodality result of Chen--Zhang~\cite{ChenZhang}.
The Ehrhart-theoretic comparison is stated only in forms directly supported by
the cited sources and by the results proved below.

\medskip
\noindent\textbf{Method and organization.}
Section~\ref{sec-2} builds the transfer-matrix framework: the double-step
matrix $C_s=B_s\bar{B}_s$ encodes two consecutive constraints at once, and
its resolvent generates $\Phi_s$, $\Psi_s$, and the cyclic trace series
$\Psi_s^{\mathrm{cyc}}$ and $\Omega_s^{\mathrm{cyc}}$ via adjugate formulas
and Jacobi's determinant identity.  The polynomial
families $P_n$ and $Q_n$---odd and even parts of $(1+\sqrt{y})^n$---appear as
the numerators and common denominator (Theorems~\ref{thm-rational}
and~\ref{thm-cyclic-intro}).  Section~\ref{sec-3} applies the same
compression idea to fixed dilations, proving Theorem~\ref{thm-Gm-general}
and the Ehrhart-theoretic consequences including Gorenstein behavior.
The concluding section collects six open problems.

\begin{table}[htbp]
\centering
\footnotesize
\caption{Structural comparison between the uniform adjacent-sum family ($c_i\equiv s$) and the alternating family ($c_i$ alternating between $s$ and $s+1$) studied here.  Citations in the uniform column indicate where the corresponding results are proved or discussed; results for the alternating family are proved in this paper unless otherwise noted.}
\label{tab-comparison}
\renewcommand{\arraystretch}{1.15}
\begin{tabularx}{\textwidth}{@{}p{2.65cm}XX@{}}
\toprule
\textbf{Aspect} & \textbf{Uniform family} & \textbf{Alternating family studied here} \\
\midrule
Model
& Constant adjacent-sum bound; equivalent to path graph-polytopes, weighted graphs, and magic labellings \cite{Bona2,Xin}.
& Bounds alternate between $s$ and $s+1$. \\[2pt]

Transfer matrix
& One-step transfer matrix; also the unit-primitive matrix model \cite{Bona2,Xin}.
& Two-step transfer matrix forced by the alternation. \\[2pt]

Dimension counts
& One rational dimension generating function, with continued-fraction and Chebyshev-type forms \cite{Bona2,Xin}.
& Odd and even dimensions split into two rational series, with M\"obius/coupled recurrences. \\[2pt]

Cyclic version
& Circular inequalities are treated in the uniform setting \cite{Bona2}.
& Both cyclic parity classes become trace formulas with the open-path denominator. \\[2pt]

Ehrhart viewpoint
& The path graph-polytope interpretation gives the same counting problem \cite{Xin}.
& Fixed-dilation dimension series recover Ehrhart data and $h^*$-coefficients. \\[2pt]

$h^*$-behavior
& In the related uniform graph-polytope setting, the Ehrhart numerator is palindromic \cite{Xin}; unimodality is known for a related zig-zag chain-polytope model \cite{ChenZhang}.
& The case $s=1$ gives all even dimensions; $(s,d)=(3,2)$ is an additional low-dimensional case; for every $s\ge2$ the property fails in some even dimension. \\
\bottomrule
\end{tabularx}
\end{table}

\section{Transfer matrices and dimension generating functions}\label{sec-2}

Throughout this section $s \ge 1$ is a fixed integer.

To orient the reader, the argument in this section proceeds in four main stages,
followed by the cyclic variant.  We first encode the alternating inequalities by a two-step transfer matrix and
rewrite the dimension counts as matrix-resolvent series (\S\ref{sec-2}.1). We
then identify the common denominator and the relevant adjugate numerators with
two polynomial families extracted from Pascal's triangle (\S\ref{sec-2}.2).
These identities yield the rational expressions for $\Phi_s$ and $\Psi_s$
(\S\ref{sec-2}.3), from which the M\"obius recurrence and the trigonometric
closed forms follow (\S\ref{subsec-main-proofs}); the cyclic variant is treated separately
in \S\ref{subsec-cyclic}.

\subsection{Transfer-matrix formulation}

For integers $0 \le i, j \le s+1$, let $a_{ij}(d,s)$ denote the number of
nonnegative integer vectors $(x_1,\dots,x_d)$ satisfying the system
\eqref{eq-poly-def} with boundary conditions $x_1 = i$ and $x_d = j$.
Write $A_{d,s} := (a_{ij}(d,s))_{0 \le i,j \le s+1}$.
Then
\begin{equation}\label{eq-Nd-sum}
N_d(s) = \sum_{i=0}^{s+1}\sum_{j=0}^{s+1} a_{ij}(d,s).
\end{equation}

Define the two \emph{transition matrices} $B_s$ and $\overline{B}_s$ of size
$(s+2)\times(s+2)$, indexed by $\{0,1,\dots,s+1\}$, by
\begin{equation}\label{eq-BbarB-def}
(B_s)_{ij} =
\begin{cases}
1, & i+j \le s,\\
0, & \text{otherwise},
\end{cases}
\qquad
(\overline{B}_s)_{ij} =
\begin{cases}
1, & i+j \le s+1,\\
0, & \text{otherwise}.
\end{cases}
\end{equation}
Here $B_s$ records transitions across an odd-indexed constraint
$x_r + x_{r+1} \le s$, and $\overline{B}_s$ records transitions across an
even-indexed constraint $x_r + x_{r+1} \le s+1$. The \emph{double-step
transfer matrix} is
\begin{equation}\label{eq-Cs-def}
C_s := B_s \overline{B}_s.
\end{equation}

\begin{proposition}\label{prop-transfer}
For every integer $m \ge 1$,
\begin{equation}\label{eq-A-even}
A_{2m,s} = C_s^{\,m-1} B_s,
\end{equation}
and
\begin{equation}\label{eq-A-odd}
A_{2m+1,s} = C_s^{\,m}.
\end{equation}
\end{proposition}

\begin{proof}
For $d=2$ the system consists only of $x_1 + x_2 \le s$, giving
$A_{2,s} = B_s$. For $d=3$ the inequalities are $x_1 + x_2 \le s$ and
$x_2 + x_3 \le s+1$, giving $A_{3,s} = B_s \overline{B}_s = C_s$.

Now suppose $A_{2m,s}=C_s^{m-1}B_s$ and $A_{2m+1,s}=C_s^m$ for some $m\ge1$.
Passing from dimension $2m$ to $2m+1$ appends the new constraint
$x_{2m}+x_{2m+1}\le s+1$, so
\[
A_{2m+1,s}=A_{2m,s}\,\overline{B}_s=C_s^{m-1}B_s\overline{B}_s=C_s^m.
\]
Passing from dimension $2m+1$ to $2m+2$ then appends the new constraint
$x_{2m+1}+x_{2m+2}\le s$, so
\[
A_{2m+2,s}=A_{2m+1,s}B_s=C_s^mB_s.
\]
This proves \eqref{eq-A-even} and \eqref{eq-A-odd} by induction.
\end{proof}

Setting $u_s := (1,1,\dots,1)^\top \in \R^{s+2}$, we may write
$N_d(s) = u_s^\top A_{d,s}\, u_s$, and the following corollary records the
resulting matrix-resolvent expressions.

\begin{corollary}\label{cor-Nd-matrix}
For every $m \ge 1$,
\[
N_{2m}(s) = u_s^\top C_s^{\,m-1} B_s\, u_s,
\qquad
N_{2m+1}(s) = u_s^\top C_s^{\,m}\, u_s.
\]
\end{corollary}

\begin{proof}
By \eqref{eq-Nd-sum}, $N_d(s)$ is the sum of all entries of $A_{d,s}$.
Equivalently, if $u_s=(1,\ldots,1)^\top$, then
\[
N_d(s)=u_s^\top A_{d,s}u_s.
\]
Substituting the two transfer-matrix formulas of
Proposition~\ref{prop-transfer}, namely
$A_{2m,s}=C_s^{m-1}B_s$ and $A_{2m+1,s}=C_s^m$, gives the two displayed
identities.
\end{proof}

\begin{lemma}[Resolvent expressions]\label{lem-resolvent}
For every $s \ge 1$,
\begin{equation}\label{eq-Phi-resolvent}
\Phi_s(y) = u_s^\top (I - y C_s)^{-1} u_s
= \frac{u_s^\top (I - y C_s)^* u_s}{\det(I - y C_s)},
\end{equation}
and
\begin{equation}\label{eq-Psi-resolvent}
\Psi_s(y) = u_s^\top (I - y C_s)^{-1} B_s\, u_s
= \frac{u_s^\top (I - y C_s)^* B_s\, u_s}{\det(I - y C_s)},
\end{equation}
where $M^*$ denotes the classical adjugate of $M$.
\end{lemma}

\begin{proof}
By Corollary~\ref{cor-Nd-matrix},
\[
\Phi_s(y)
= \sum_{m \ge 0} u_s^\top C_s^m\, u_s \cdot y^m
= u_s^\top \Bigl(\sum_{m \ge 0} (y C_s)^m\Bigr) u_s
= u_s^\top (I - y C_s)^{-1} u_s.
\]
The adjugate form follows from $(I - y C_s)^{-1} = (I - y C_s)^*/\det(I - y C_s)$.
The identity for $\Psi_s$ is obtained in the same way.
\end{proof}

\begin{remark}[Normalization convention]\label{rem-N1-convention}
The generating function $\Phi_s(y)$ uses the auxiliary normalization
$\widetilde N_1(s):=s+2$ as its constant term, chosen precisely so that the
resolvent identity $\Phi_s(y)=u_s^\top(I-yC_s)^{-1}u_s$ holds
(the $m=0$ term of $u_s^\top C_s^0 u_s = \|u_s\|_1 = s+2$).
\emph{This differs from the actual one-dimensional lattice-point count}
$\#([0,s]\cap\Z)=s+1$; the series $\Phi_s(y)$ is therefore \emph{not}
the generating function of $\{N_d(s)\}_{d\ge1}$ starting from $d=1$.
To avoid conflating these two quantities, we reserve $L_1^{(s)}(m)=ms+1$
for the Ehrhart-theoretic convention used in Section~\ref{sec-3}.
\end{remark}

\subsection{The polynomial families \texorpdfstring{$P_n$}{Pn} and \texorpdfstring{$Q_n$}{Qn}}

The polynomials $P_n(-y)$ and $Q_n(-y)$ introduced in~\eqref{eq-PQ-def}
will be identified in \S\ref{sec-2}.3 as the adjugate numerators
$u_s^\top(I-yC_s)^*u_s$ and $u_s^\top(I-yC_s)^*B_su_s$, and as the
denominator $\det(I-yC_s)$, of the resolvents
\eqref{eq-Phi-resolvent}--\eqref{eq-Psi-resolvent}.  We collect here the
algebraic and trigonometric properties of these families that will be used
throughout the section.

\begin{lemma}[First-order system]\label{lem-PQ-rec1}
The pair $(P_n(-y),\, Q_n(-y))$ satisfies the \emph{first-order} system
(each member expressed in terms of both members at the previous index):
\begin{align}
P_n(-y) &= P_{n-1}(-y) + Q_{n-1}(-y),\label{eq-P-rec1}\\
Q_n(-y) &= Q_{n-1}(-y) - y\, P_{n-1}(-y).\label{eq-Q-rec1}
\end{align}
\end{lemma}

\begin{proof}
Both identities are immediate from Pascal's identity:
$\binom{n}{2k+1} = \binom{n-1}{2k+1} + \binom{n-1}{2k}$ gives \eqref{eq-P-rec1}
(the $k$-th coefficient of $P_n(-y)$ splits as the $k$-th coefficient of
$P_{n-1}(-y)$ plus the $k$-th coefficient of $Q_{n-1}(-y)$), and
$\binom{n}{2k} = \binom{n-1}{2k} + \binom{n-1}{2k-1}$
gives \eqref{eq-Q-rec1} after reindexing the summation variable
($k \mapsto k-1$ in the second sum introduces a factor of $-y$).
\end{proof}

\begin{remark}
The first-order system will be used again in Lemma~\ref{lem-tan-ratio} to derive the tangent ratio identity and later in the numerator computations of Section~\ref{sec-2}.3.
\end{remark}

\begin{lemma}[Second-order recurrence]\label{lem-PQ-rec2}
The sequences $P_n(-y)$ and $Q_n(-y)$ each satisfy, for $n \ge 2$, the
\emph{second-order} linear recurrence:
\begin{align}
P_n(-y) &= 2P_{n-1}(-y) - (1+y)P_{n-2}(-y),\label{eq-P-rec2}\\
Q_n(-y) &= 2Q_{n-1}(-y) - (1+y)Q_{n-2}(-y).\label{eq-Q-rec2}
\end{align}
\end{lemma}

\begin{proof}
We derive both identities from Lemma~\ref{lem-PQ-rec1}.
Applying \eqref{eq-P-rec1} twice gives
\[
P_n(-y) = P_{n-1}(-y) + Q_{n-1}(-y).
\]
From \eqref{eq-Q-rec1} at step $n-1$: $Q_{n-1}(-y) = Q_{n-2}(-y) - yP_{n-2}(-y)$.
Substituting and using \eqref{eq-P-rec1} at step $n-1$ to write
$Q_{n-2}(-y) = P_{n-1}(-y) - P_{n-2}(-y)$:
\[
P_n(-y) = P_{n-1}(-y) + P_{n-1}(-y) - P_{n-2}(-y) - yP_{n-2}(-y)
= 2P_{n-1}(-y) - (1+y)P_{n-2}(-y).
\]
The identity \eqref{eq-Q-rec2} follows analogously by eliminating $P_{n-1}$
from the system \eqref{eq-P-rec1}--\eqref{eq-Q-rec1}.
\end{proof}

\begin{remark}[Chebyshev connection]\label{rem-chebyshev}
Recall that the Chebyshev polynomials of the first and second kinds are defined by
\[
T_n(\cos\theta)=\cos(n\theta),
\qquad
U_{n-1}(\cos\theta)=\frac{\sin(n\theta)}{\sin\theta}
\qquad (n\ge1).
\]
From the decomposition $(1 + \sqrt{y})^n = Q_n(y) + \sqrt{y}\,P_n(y)$, setting
$y = -\tan^2\theta$ gives
$(1 + i\tan\theta)^n = Q_n(-\tan^2\theta) + i\tan\theta\cdot P_n(-\tan^2\theta)$.
Since $1 + i\tan\theta = e^{i\theta}/\cos\theta$, taking real and imaginary parts
yields the exact identities
\begin{equation}\label{eq-QP-trig}
Q_n(-\tan^2\theta) = \frac{\cos(n\theta)}{\cos^n\theta},
\qquad
P_n(-\tan^2\theta) = \frac{\sin(n\theta)}{\sin\theta\,\cos^{n-1}\theta}.
\end{equation}
Equivalently,
\[
Q_n(-\tan^2\theta)=\frac{T_n(\cos\theta)}{\cos^n\theta},
\qquad
P_n(-\tan^2\theta)=\frac{U_{n-1}(\cos\theta)}{\cos^{n-1}\theta}.
\]
Thus $Q_n(-\tan^2\theta)$ is a rescaled Chebyshev polynomial of the first kind,
while $P_n(-\tan^2\theta)$ is the corresponding rescaled Chebyshev polynomial of
the second kind. These identities underlie the determinant formula in
Lemma~\ref{lem-det} and the tangent ratio in Lemma~\ref{lem-tan-ratio}.
\end{remark}

\begin{lemma}[Tangent ratio]\label{lem-tan-ratio}
Let $y = \tan^2 x$. Then for every integer $n \ge 1$,
\begin{equation}\label{eq-tan-ratio}
\frac{P_n(-y)}{Q_n(-y)} = \frac{\tan(nx)}{\tan x}.
\end{equation}
\end{lemma}

\begin{proof}
We argue by induction on $n$. The case $n=1$ is immediate since
$P_1(-y)=1$, $Q_1(-y)=1$, and $\tan(x)/\tan(x)=1$. By Lemma~\ref{lem-PQ-rec1},
\[
\frac{P_n(-y)}{Q_n(-y)}
= \frac{\rho_{n-1} + 1}{1 - y \rho_{n-1}},
\qquad
\rho_{n-1} := \frac{P_{n-1}(-y)}{Q_{n-1}(-y)}.
\]
By the induction hypothesis $\rho_{n-1} = \tan((n-1)x)/\tan x$ and $y = \tan^2 x$, so
\[
\frac{P_n(-y)}{Q_n(-y)}
= \frac{\tfrac{\tan((n-1)x)}{\tan x} + 1}{1 - \tan^2\!x\cdot\tfrac{\tan((n-1)x)}{\tan x}}
= \frac{1}{\tan x}\cdot
  \frac{\tan((n-1)x) + \tan x}{1 - \tan((n-1)x)\tan x}
= \frac{\tan(nx)}{\tan x},
\]
by the addition formula for tangent.

Alternatively, the identity follows directly from \eqref{eq-QP-trig}: if
$y=\tan^2 x$, then
\[
\frac{P_n(-y)}{Q_n(-y)}
=
\frac{\sin(nx)/(\sin x\,\cos^{n-1}x)}{\cos(nx)/\cos^n x}
=\frac{\sin(nx)\cos x}{\sin x\cos(nx)}
=\frac{\tan(nx)}{\tan x}.
\]
\end{proof}

\subsection{Evaluation of the denominator and numerators}

\begin{lemma}\label{lem-det}
For every integer $s \ge 0$,
\begin{equation}\label{eq-det-main}
\det(I - y C_s) = Q_{s+2}(-y).
\end{equation}
\end{lemma}

\begin{proof}
Set $M_s:=I-yC_s$. We first derive the explicit form of $M_s$ from the
definition of $C_s = B_s\overline{B}_s$.

\medskip
\noindent\textit{Step 1: Computing $(C_s)_{ij}$.}
By definition,
\[
(C_s)_{ij}
= \sum_{k=0}^{s+1}(B_s)_{ik}(\overline{B}_s)_{kj}
= \#\{k \ge 0 : i+k \le s \text{ and } k+j \le s+1\}
= \bigl(\min(s-i,\, s+1-j)+1\bigr)_+,
\]
where $(u)_+ := \max(u,0)$. The identity
$\min(s-i,\,s+1-j) = s - \max(i,\,j-1)$
holds because $s-i \le s+1-j \iff j \le i+1$.  Therefore
\begin{equation}\label{eq-Cij}
(C_s)_{ij} = \bigl(s+1-\max(i,j-1)\bigr)_+
\qquad (0 \le i,j \le s+1).
\end{equation}
In particular: the last row ($i = s+1$) is identically zero; for $0 \le i \le s$ the
entry equals $s-i+1$ when $j \le i+1$, and equals $s+2-j$ when $j > i+1$.  This
gives the staircase structure
\begin{equation}\label{eq-Ms-explicit}
M_s=
\begin{pmatrix}
1-(s+1)y & -(s+1)y & -sy & \cdots & -2y & -y\\
-sy & 1-sy & -sy & \cdots & -2y & -y\\
-(s-1)y & -(s-1)y & 1-(s-1)y & \cdots & -2y & -y\\
\vdots & \vdots & \vdots & \ddots & \vdots & \vdots\\
-y & -y & -y & \cdots & 1-y & -y\\
0 & 0 & 0 & \cdots & 0 & 1
\end{pmatrix},
\end{equation}
where the $(i,j)$-entry is $\delta_{ij} - (s+1-\max(i,j-1))_+ \cdot y$.

\medskip
\noindent\textit{Step 2: Row and column reduction.}
We apply the following two determinant-preserving operations to
\eqref{eq-Ms-explicit}, in the order stated:
\begin{enumerate}[label=(\roman*),leftmargin=2em]
\item \emph{Column operation}: replace column~$1$ by (column~$1$)$-$(column~$2$).
After this operation, the $(i,1)$-entry of the resulting matrix $M_s'$ is
$(M_s)_{i1}-(M_s)_{i2}$. Using \eqref{eq-Ms-explicit},
\begin{align*}
(M_s')_{11} &= \bigl(1-(s+1)y\bigr)-\bigl(-(s+1)y\bigr) = 1,\\
(M_s')_{21} &= -sy-(1-sy) = -1,\\
(M_s')_{i1} &= 0\quad\text{for }i\ge 3.
\end{align*}
Columns $2,\dots,s+2$ are unchanged.
\item \emph{Row operation}: replace row~$1$ by (row~$1$)$-$(row~$2$).
Row~$1$ of $M_s'$ is
\[
(1,\,-(s+1)y,\,-sy,\,\dots,\,-2y,\,-y),
\]
and row~$2$ is
\[
(-1,\,1-sy,\,-sy,\,\dots,\,-2y,\,-y).
\]
Subtracting gives
\[
\widetilde r_1
=\bigl(2,\,-(1+y),\,0,\,\dots,\,0\bigr).
\]
\end{enumerate}
The result is the matrix
\[
\widetilde M_s=
\begin{pmatrix}
2 & -(1+y) & 0 & \cdots & 0 & 0\\
-1 & 1-sy & -sy & \cdots & -2y & -y\\
0 & -(s-1)y & 1-(s-1)y & \cdots & -2y & -y\\
\vdots & \vdots & \vdots & \ddots & \vdots & \vdots\\
0 & -y & -y & \cdots & 1-y & -y\\
0 & 0 & 0 & \cdots & 0 & 1
\end{pmatrix}.
\]

\medskip
\noindent\textit{Step 3: Recursion.}
Expanding along the first row of $\widetilde M_s$,
\begin{equation}\label{eq-detCs}
\det(M_s)=2\det(M_{s-1})-(1+y)\det(M_{s-2}), \qquad s\ge2.
\end{equation}
The $(1,1)$ cofactor deletes row and column $1$ from $\widetilde M_s$, leaving the
$(s+1)\times(s+1)$ submatrix occupying rows/columns $2,\ldots,s+2$; a direct
inspection of \eqref{eq-Ms-explicit} shows that this submatrix is exactly
$M_{s-1}$ after relabeling indices $i \mapsto i-1$ (the staircase structure
\eqref{eq-Cij} is preserved under this relabeling).

The $(1,2)$ cofactor deletes row $1$ and column $2$ from $\widetilde M_s$. The
resulting $s\times s$ submatrix (rows $2,\ldots,s+2$, columns $1,3,4,\ldots,s+2$) may be relabeled to match $M_{s-2}$: its first column comes from the entry $-1$ in row $2$ of $\widetilde M_s$, while the remaining columns inherit the same staircase pattern as in \eqref{eq-Ms-explicit}. Thus the cofactor contributes exactly the $M_{s-2}$ term, and the combined effect of the prefactor $-(1+y)$ and the cofactor sign gives the coefficient $-(1+y)$ in \eqref{eq-detCs}.

By Lemma~\ref{lem-PQ-rec2}, the sequence $Q_{s+2}(-y)$ satisfies the same
recurrence. The initial values agree as well, since
\[
\det(M_0)=1-y=Q_2(-y), \qquad \det(M_1)=1-3y=Q_3(-y),
\]
which are verified directly from \eqref{eq-Ms-explicit} for $s=0,1$.
Therefore $\det(M_s)=Q_{s+2}(-y)$ for all $s\ge0$.
\end{proof}

The next lemma identifies the column sums of the adjugate and is the main
input for the two numerator computations that follow.

\begin{lemma}\label{lem-col-adj-Q}
For every integer $s \ge 0$ and each $1 \le i \le s+2$,
\begin{equation}\label{eq-col-sum-adj}
u_s^\top (I - yC_s)^* e_i = Q_{i-1}(-y),
\end{equation}
where $e_i$ denotes the $i$-th standard basis vector of $\R^{s+2}$.
\end{lemma}

\begin{proof}
Let $M_s:=I-yC_s$, and let $D_{s,i}$ be the matrix obtained from $M_s$ by
replacing its $i$-th row with $u_s^\top$. By the defining property of the
adjugate (Cramer's rule),
\[
u_s^\top (I-yC_s)^*e_i=\det(D_{s,i}).
\]
Thus it suffices to compute $\det(D_{s,i})$.

\medskip
\noindent\textit{Row and column operations on $D_{s,i}$.}
We simplify $D_{s,i}$ in two steps.
\begin{enumerate}[leftmargin=2em,label=(\roman*)]
\item For each row $k = i+1, \ldots, s+1$: add $(s+2-k)y$ times row~$i$
  (the all-ones row) to row~$k$, and leave the last row $k=s+2$ unchanged.
  Using \eqref{eq-Ms-explicit}, for $i+1\le k\le s+1$ and $j\le k-1$ one has
  $(M_s)_{kj}=-(s+2-k)y$, while $(M_s)_{kk}=1-(s+2-k)y$. After the row operation,
  these entries become
  \[
  -(s+2-k)y+(s+2-k)y=0\qquad (j\le k-1),
  \]
  and
  \[
  1-(s+2-k)y+(s+2-k)y=1\qquad (j=k).
  \]
  Thus, after step~(i), every row $k\ge i+1$ has zeros in columns $1,\ldots,k-1$,
  and the lower-right block indexed by rows and columns $i+1,\ldots,s+2$
  is upper triangular with diagonal entries equal to $1$ (the final row is already
  that of the identity matrix).
\item For $j = 2, \ldots, i$: subtract column~$j$ from column~$j-1$.  Because after
  step~(i) the entries in rows $i+1,\ldots,s+2$ and columns $1,\ldots,i$ are already
  zero, these column operations affect only the upper-left $({i-1})\times({i-1})$
  block and turn it into the matrix with leading subblock
  \[
  A_{i-1}(y):=
  \begin{pmatrix}
  1 & -y & -y & -y & \cdots & -y & -y\\
  -1 & 1 & -y & -y & \cdots & -y & -y\\
  0 & -1 & 1 & -y & \cdots & -y & -y\\
  0 & 0 & -1 & 1 & \cdots & -y & -y\\
  \vdots & \vdots & \vdots & \vdots & \ddots & \vdots & \vdots\\
  0 & 0 & 0 & 0 & \cdots & 1 & -y\\
  0 & 0 & 0 & 0 & \cdots & -1 & 1
  \end{pmatrix}\in\R^{(i-1)\times(i-1)},
  \]
  that is,
  \[
  (A_{i-1}(y))_{ab}=
  \begin{cases}
  1, & a=b,\\
  -1, & a=b+1,\\
  -y, & a<b,\\
  0, & a>b+1,
  \end{cases}
  \qquad 1\le a,b\le i-1.
  \]
  The lower-right block remains the identity.
\end{enumerate}
After both operations, $D_{s,i}$ is block upper triangular with
upper-left block $A_{i-1}(y)$ (of size $(i{-}1)\times(i{-}1)$) and
lower-right block $I_{s+3-i}$; hence $\det(D_{s,i}) = \det A_{i-1}(y)$.

\medskip
\noindent\textit{Evaluating $\det A_r(y)$.}
Write $a_r(y):=\det A_r(y)$ for $r\ge1$, and set $a_0(y):=1$.
Expanding the first row of $A_r(y)$ along columns $1,2,\ldots,r$:
\[
a_r(y)=a_{r-1}(y)-y\sum_{j=2}^{r}a_{j-2}(y)\qquad (r\ge2).
\]
The term $a_{r-1}(y)$ comes from the $(1,1)$ cofactor; for the $(1,j)$ cofactor
($j \ge 2$), deleting row~$1$ and column~$j$ leaves a block upper-triangular
matrix whose upper-left block is $A_{j-2}(y)$ (the first $j-2$ rows and columns
retain the staircase pattern) and whose lower-right block is the identity
(columns $j+1,\ldots,r$ are unchanged by the column subtractions). Hence the
$(1,j)$ cofactor equals $(-1)^{1+j}(-1)^{j-2}a_{j-2}(y) = -a_{j-2}(y)$, giving
the formula above.
Subtracting the corresponding relation for $a_{r-1}(y)$ yields
\[
a_r(y)=2a_{r-1}(y)-(1+y)a_{r-2}(y)\qquad (r\ge2).
\]
The initial values are $a_0(y)=1$ and $a_1(y)=1$. By
Lemma~\ref{lem-PQ-rec2}, the polynomials $Q_r(-y)$ satisfy the same recurrence
with the same initial values. Hence $a_r(y)=Q_r(-y)$ for all $r\ge0$.
Taking $r=i-1$ gives $u_s^\top(I-yC_s)^*e_i = Q_{i-1}(-y)$.
\end{proof}

\begin{lemma}\label{lem-num-Phi}
For every integer $s \ge 0$,
\begin{equation}\label{eq-num-Phi}
u_s^\top (I - y C_s)^*\, u_s = P_{s+2}(-y).
\end{equation}
\end{lemma}

\begin{proof}
Since $u_s=\sum_{i=1}^{s+2}e_i$, we sum the column identities from Lemma~\ref{lem-col-adj-Q}:
\[
u_s^\top (I-yC_s)^*u_s
=\sum_{i=1}^{s+2}Q_{i-1}(-y).
\]
On the other hand, the first-order recurrence \eqref{eq-P-rec1} can be written
as
\[
P_m(-y)-P_{m-1}(-y)=Q_{m-1}(-y), \qquad P_0(-y)=0.
\]
Summing from $m=1$ to $m=s+2$ gives
$P_{s+2}(-y)=\sum_{i=1}^{s+2}Q_{i-1}(-y)$,
completing the proof.
\end{proof}

We record a simple summation identity needed in the proof of
Lemma~\ref{lem-num-Psi}.

\begin{lemma}[Summation identity]\label{lem-sumPQ}
For every integer $t \ge 0$,
\begin{equation}\label{eq-sumPQ}
\sum_{i=1}^{t+2} P_i(-y) = \frac{1 - Q_{t+3}(-y)}{y}.
\end{equation}
\end{lemma}

\begin{proof}
From the first-order recurrence \eqref{eq-Q-rec1},
\[
Q_m(-y)-Q_{m-1}(-y)=-y\,P_{m-1}(-y),
\quad\text{so}\quad
P_{m-1}(-y)=\frac{Q_{m-1}(-y)-Q_m(-y)}{y}.
\]
Summing for $m=2,3,\dots,t+3$ telescopes to
\[
\sum_{i=1}^{t+2}P_i(-y) = \frac{Q_1(-y)-Q_{t+3}(-y)}{y}.
\]
Since $Q_1(-y)=1$, this is \eqref{eq-sumPQ}.
\end{proof}

\begin{lemma}\label{lem-num-Psi}
For every integer $s \ge 1$,
\begin{equation}\label{eq-num-Psi}
u_s^\top (I - y C_s)^*\, B_s\, u_s = \frac{1 - Q_{s+2}(-y)}{y}.
\end{equation}
\end{lemma}

\begin{proof}
The $j$-th column of $B_s$ (for $j=0,1,\dots,s+1$, using $0$-based indexing)
contains $1$'s in positions $0,1,\ldots,s-j$ and $0$'s below (by \eqref{eq-BbarB-def}
with $i+j \le s$), so the column sum is $s-j+1 = s+2-(j+1)$. In 1-based indexing
$j' = j+1 \in \{1,\ldots,s+2\}$, the $j'$-th column sum is $s+2-j'$. Therefore
\[
u_s^\top (I-yC_s)^*B_su_s
=\sum_{j'=1}^{s+2}(s+2-j')\,Q_{j'-1}(-y).
\]
Changing the order of summation:
\[
\sum_{j'=1}^{s+2}(s+2-j')Q_{j'-1}(-y)
=\sum_{t=1}^{s+1}\sum_{j'=1}^{t}Q_{j'-1}(-y)
=\sum_{t=1}^{s+1}P_t(-y),
\]
where the last step uses the telescoping identity
\[
P_t(-y)=\sum_{j'=1}^{t}Q_{j'-1}(-y),
\]
which follows immediately from the first-order recurrence \eqref{eq-P-rec1}.
Applying Lemma~\ref{lem-sumPQ} with $t$ replaced by $s-1$:
$\sum_{t=1}^{s+1}P_t(-y) = (1-Q_{s+2}(-y))/y$,
proving \eqref{eq-num-Psi}.
\end{proof}

\subsection{Proofs of the main theorems}\label{subsec-main-proofs}

\begin{proof}[Proof of Theorem~\ref{thm-rational}]
Substituting Lemmas~\ref{lem-det} and \ref{lem-num-Phi} into formula
\eqref{eq-Phi-resolvent} gives the first formula in \eqref{eq-Phi-main}.
Lemmas~\ref{lem-det} and \ref{lem-num-Psi} yield the second.

For \eqref{eq-F-main}, decompose $F_s(x)$ by parity.
The even-dimensional terms contribute
\[
\sum_{r\ge1}N_{2r}(s)\,x^{2r}=x^2\Psi_s(x^2),
\]
since $\Psi_s(y)=\sum_{m\ge0}N_{2m+2}(s)\,y^m$ by definition and its $m=0$
term already gives the contribution of $d=2$. The odd-dimensional terms for
$d\ge3$ contribute
\[
\sum_{r\ge1}N_{2r+1}(s)\,x^{2r+1}=x\bigl(\Phi_s(x^2)-\widetilde N_1(s)\bigr),
\]
since $\Phi_s(y)=\widetilde N_1(s)+\sum_{m\ge1}N_{2m+1}(s)\,y^m$ and
$\widetilde N_1(s)=s+2$. Adding these two contributions gives \eqref{eq-F-main}.
\end{proof}

\begin{proof}[Proof of Theorem~\ref{thm-recurrence}]
We first prove \eqref{eq-Phi-rec-main}. For $s\ge 2$, write
$\Phi_s(y)=P_{s+2}(-y)/Q_{s+2}(-y)$
and apply Lemma~\ref{lem-PQ-rec1} with $n=s+2$:
\[
\Phi_s(y)
=\frac{P_{s+1}(-y)+Q_{s+1}(-y)}{Q_{s+1}(-y)-y\,P_{s+1}(-y)}.
\]
Dividing numerator and denominator by $Q_{s+1}(-y)$ gives
\[
\Phi_s(y)=\frac{1+\Phi_{s-1}(y)}{1-y\,\Phi_{s-1}(y)}.
\]
For $s=1$, the same identity is verified directly from
\[
\Phi_0(y)=\frac{P_2(-y)}{Q_2(-y)}=\frac{2}{1-y},
\qquad
\Phi_1(y)=\frac{P_3(-y)}{Q_3(-y)}=\frac{3-y}{1-3y}.
\]
Thus \eqref{eq-Phi-rec-main} holds for all $s\ge1$, and iterating it yields the
continued-fraction expansion.

We next prove \eqref{eq-Psi-rec-main}. By Theorem~\ref{thm-rational},
$1+y\Psi_s(y)=1/Q_{s+2}(-y)$.
Using Lemma~\ref{lem-PQ-rec1}: $Q_{s+2}(-y)=Q_{s+1}(-y)-yP_{s+1}(-y)$, so
\[
1+y\Psi_s(y)
=\frac{\tfrac{1}{Q_{s+1}(-y)}}{1-y\,\tfrac{P_{s+1}(-y)}{Q_{s+1}(-y)}}
=\frac{1+y\Psi_{s-1}(y)}{1-y\Phi_{s-1}(y)}.
\]
This is \eqref{eq-Psi-rec-main}. The initial value $\Psi_0(y)=1/(1-y)$ follows from
\[
\Psi_0(y)=\frac{1-Q_2(-y)}{yQ_2(-y)}=\frac{1-(1-y)}{y(1-y)}=\frac{1}{1-y}.
\]
\end{proof}

\begin{remark}[M\"obius structure of the recurrence]\label{rem-mobius}
The map $T(f) = (1+f)/(1-yf)$ is a M\"obius transformation in the variable $f$,
with parameter $y$. Its fixed points satisfy $yf^2 = -1$, giving
$f = \pm\, i/\!\sqrt{y}$. The sequence
$\{\Phi_s(y)\}_{s \ge 0}$ is the orbit of $\Phi_0(y) = 2/(1-y)$
under iteration of $T$. The pole $y^*$ corresponds precisely to the value of
$y$ at which the orbit escapes to infinity, i.e.,
$(s+2)\arctan\sqrt{y} = \pi/2$.
\end{remark}

\begin{proof}[Proof of Theorem~\ref{thm-tangent}]
Set $x=\arctan\sqrt{y}$, so $y=\tan^2x$ and $\sqrt{y}=\tan x$.
For $\Phi_s(y)$: Theorem~\ref{thm-rational} and Lemma~\ref{lem-tan-ratio} give
\[
\Phi_s(y)
=\frac{P_{s+2}(-y)}{Q_{s+2}(-y)}
=\frac{\tan((s+2)x)}{\tan x}
=\frac{\tan((s+2)\arctan\sqrt{y})}{\sqrt{y}}.
\]
For $\Psi_s(y)$: Theorem~\ref{thm-rational} and identity \eqref{eq-QP-trig} give
\[
\Psi_s(y)
=\frac{1}{y}\!\left(\frac{1}{Q_{s+2}(-y)}-1\right)
=\frac{1}{y}\!\left(\frac{\cos^{s+2}x}{\cos((s+2)x)}-1\right).
\]
Substituting $x=\arctan\sqrt{y}$ yields \eqref{eq-Psi-trig-main}.
\end{proof}

\begin{remark}[Euler tangent numbers and asymptotic growth]\label{rem-euler}
The closed form \eqref{eq-Phi-tan-main} connects $N_{2m+1}(s)$ to the
Euler tangent numbers $T_{2k+1}$ (OEIS A000182 \cite{OEIS}).
Both $\Phi_s(y)$ and $\Psi_s(y)$ share the denominator $Q_{s+2}(-y)$, whose
dominant root is
\[
y^*:=\tan^2\!\Bigl(\frac{\pi}{2(s+2)}\Bigr).
\]
Recall the notation from Theorem~\ref{thm-tangent}:
$\alpha_s=\pi/(2(s+2))$, $y^*=\tan^2\alpha_s$, and the constants
$\kappa_s(y^*)$, $\lambda_s(y^*)$ are defined in
\eqref{eq-kappa-lambda-def-intro}.  These constants are obtained by computing
residues of $\Phi_s$ and $\Psi_s$ at $y = y^*$.
Since $Q_{s+2}(-y)$ has a simple zero at $y^*$, we differentiate the identity
\[
Q_{s+2}(-\tan^2\theta)=\frac{\cos((s+2)\theta)}{\cos^{s+2}\!\theta}
\]
from \eqref{eq-QP-trig} at $\theta=\alpha_s$. Using
$y=\tan^2\theta$ and $dy/d\theta=2\tan\theta\sec^2\theta$, we obtain
\[
\left.\frac{d}{dy}Q_{s+2}(-y)\right|_{y=y^*}
=\frac{-(s+2)/\cos^{s+2}\!\alpha_s}{2\sin\alpha_s/\cos^3\!\alpha_s}
=-\frac{s+2}{2\sin\alpha_s\cos^{s-1}\!\alpha_s}.
\]
The residue formula then yields precisely the constants
\[
\kappa_s(y^*) = \frac{2\sec^2\!\alpha_s}{(s+2)\,y^*},
\qquad
\lambda_s(y^*)
= \frac{2\cos^{s+3}\!\alpha_s}{(s+2)\sin^3\!\alpha_s},
\]
and hence the asymptotics stated in \eqref{eq-open-asymptotics-intro}.

For $s=1$, $y^*=1/3$, giving the exact formulas
$N_{2m+1}(1)=8\cdot 3^{m-1}$ ($m\ge 1$) and $N_{2m+2}(1)=3^{m+1}$
(which follow from the partial-fraction expansions of $\Phi_1$ and $\Psi_1$).
For $s=2$, $y^*=3-2\sqrt2$, giving exponential growth rate $(3+2\sqrt2)^m$.
\end{remark}

\subsection{Verification in small cases}\label{subsec-small-cases}

\begin{example}\label{ex-small-cases}
We record the first few expansions and the corresponding lattice-point counts,
and verify the pole locations predicted by Theorem~\ref{thm-tangent}.

\medskip
\noindent\textit{Case $s = 1$.}\quad
\[
\Phi_1(y) = \frac{3-y}{1-3y},
\qquad
\Psi_1(y)=\frac{3}{1-3y}.
\]
Thus
$N_1(1)=3,\; N_3(1)=8,\; N_5(1)=24,\; N_7(1)=72$,
and
$N_2(1)=3,\; N_4(1)=9,\; N_6(1)=27,\; N_8(1)=81$.
The dominant pole is $y^*=1/3=\tan^2(\pi/6)$, so both subsequences
grow geometrically with ratio $3$.

\medskip
\noindent\textit{Case $s = 2$.}\quad
\[
\Phi_2(y) = \frac{4-4y}{1-6y+y^2},
\qquad
\Psi_2(y)=\frac{6-y}{1-6y+y^2}.
\]
Thus
$N_1(2)=4,\; N_3(2)=20,\; N_5(2)=116,\; N_7(2)=676$,
and
$N_2(2)=6,\; N_4(2)=35,\; N_6(2)=204,\; N_8(2)=1189$.
The dominant pole is $y^*=3-2\sqrt2=\tan^2(\pi/8)$, giving growth rate $(3+2\sqrt2)^m$.

\medskip
\noindent\textit{Case $s = 3$.}\quad
\[
\Phi_3(y) = \frac{5-10y+y^2}{1-10y+5y^2},
\qquad
\Psi_3(y)=\frac{10-5y}{1-10y+5y^2}.
\]
Hence
$N_1(3)=5,\; N_3(3)=40,\; N_5(3)=376,\; N_7(3)=3560$,
and
$N_2(3)=10,\; N_4(3)=95,\; N_6(3)=900,\; N_8(3)=8525$.
All values agree with direct transfer-matrix computations.
\end{example}

\begin{remark}[OEIS connections]\label{rem-oeis-intro}
The polynomial families $P_n$ and $Q_n$ appear in OEIS as triangles A034867 and A034839.
The coefficients of $\Phi_1$ coincide with A080923 (up to indexing) and those
of $\Phi_2$ with A077445 \cite{OEIS}, providing a new geometric source for these sequences.
\end{remark}

\subsection{Cyclic variants}\label{subsec-cyclic}

We close Section~\ref{sec-2} by proving Theorems~\ref{thm-cyclic-intro},
\ref{thm-cyclic-recurrence}, and~\ref{thm-cyclic-trig}.
The cyclic closure is obtained by imposing the additional constraint
$x_d+x_1\le s+\delta_d$, consistently with the alternating pattern
($\delta_d=0$ for $d$ odd and $\delta_d=1$ for $d$ even).  Thus
\[
N_d^{\mathrm{cyc}}(s):=
\#\Bigl(\bigl\{x\in\cP_d^{(s)}:x_d+x_1\le s+\delta_d\bigr\}\cap\Z^d\Bigr).
\]
By Proposition~\ref{prop-transfer}, the matrix responsible for the closing step
is $B_s$ in odd dimension and $\overline B_s$ in even dimension.  Since
$C_s=B_s\overline B_s$, we obtain
\begin{equation}\label{eq-Ncyc-trace}
N_{2m}^{\mathrm{cyc}}(s)=\operatorname{Tr}(C_s^{\,m}),
\qquad
N_{2m+1}^{\mathrm{cyc}}(s)=\operatorname{Tr}(C_s^{\,m}B_s),
\qquad m\ge1.
\end{equation}
Consequently
\begin{equation}\label{eq-Psicyc-resolvent}
\Psi_s^{\mathrm{cyc}}(y)
=\operatorname{Tr}\bigl(C_s(I-yC_s)^{-1}\bigr),
\end{equation}
and
\begin{equation}\label{eq-Omegacyc-resolvent}
\Omega_s^{\mathrm{cyc}}(y)
=\operatorname{Tr}\bigl(C_s(I-yC_s)^{-1}B_s\bigr).
\end{equation}

\medskip
\noindent\textit{Even and odd cyclic series.}
The even-dimensional series follows from Jacobi's formula.  The odd-dimensional
series requires one additional adjugate trace, which is still explicit.

\begin{lemma}[Derivative identity]\label{lem-Qderiv}
For every integer $n\ge1$,
\begin{equation}\label{eq-Qderiv}
\frac{d}{dy}Q_n(-y)=-\frac{n}{2}P_{n-1}(-y).
\end{equation}
\end{lemma}

\begin{proof}
Differentiating
$Q_n(-y)=\sum_{k=0}^{\lfloor n/2\rfloor}\binom{n}{2k}(-y)^k$
termwise gives
\[
\frac{d}{dy}Q_n(-y)
=\sum_{k=1}^{\lfloor n/2\rfloor}(-1)^k k\binom{n}{2k}y^{k-1}.
\]
Using $k\binom{n}{2k}=\frac{n}{2}\binom{n-1}{2k-1}$ and then setting
$j=k-1$ gives
\[
\frac{d}{dy}Q_n(-y)
=-\frac{n}{2}\sum_{j=0}^{\lfloor(n-2)/2\rfloor}(-1)^j
\binom{n-1}{2j+1}y^j
=-\frac{n}{2}P_{n-1}(-y).
\]
\end{proof}

We shall need a more explicit form of the numerator appearing in the odd
cyclic series.  The proof is slightly longer than the even cyclic case,
because the odd closing condition introduces a non-symmetric adjugate trace
rather than a logarithmic derivative.  We separate the computation into a few
small steps.

\begin{lemma}[A directional derivative identity]\label{lem-R-directional}
Let
\[
M_s(y):=I-yC_s .
\]
Then
\[
\mathcal R_s(y):=\operatorname{Tr}\bigl(M_s(y)^*B_s\bigr)
=
-\left.\frac{\partial}{\partial t}
\det\bigl(M_s(y)-tB_s\bigr)\right|_{t=0}.
\]
\end{lemma}

\begin{proof}
For any square matrix $A$ and any perturbation $E$, Jacobi's formula gives
\[
\left.\frac{d}{dt}\det(A+tE)\right|_{t=0}
=\operatorname{Tr}(A^*E).
\]
Taking $A=M_s(y)$ and $E=-B_s$ gives
\[
\left.\frac{\partial}{\partial t}
\det\bigl(M_s(y)-tB_s\bigr)\right|_{t=0}
=-\operatorname{Tr}\bigl(M_s(y)^*B_s\bigr),
\]
which proves the claim.
\end{proof}

\begin{lemma}[Reduction to an anti-diagonal adjugate sum]\label{lem-R-antidiagonal-reduction}
Put $W=1+y$, and let $T_{s+1}(y)$ be the $(s+1)\times(s+1)$ tridiagonal matrix
\[
T_{s+1}(y)=
\begin{pmatrix}
2 & -W & 0 & \cdots & 0\\
-1 & 2 & -W & \ddots & \vdots\\
0 & -1 & 2 & \ddots & 0\\
\vdots & \ddots & \ddots & \ddots & -W\\
0 & \cdots & 0 & -1 & 1-y
\end{pmatrix}.
\]
Let
\[
A^{(s)}_{p,q}(y):=\bigl(T_{s+1}(y)^*\bigr)_{p,q},
\qquad 1\le p,q\le s+1.
\]
Then
\begin{equation}\label{eq-R-antidiagonal-new}
\mathcal R_s(y)
=
\sum_{i=1}^{s+1} A^{(s)}_{s+2-i,i}(y)
-
\sum_{i=1}^{s} A^{(s)}_{s+1-i,i}(y).
\end{equation}
\end{lemma}

\begin{proof}
Let $n=s+2$.  We write matrices in $0$-based indices in this paragraph.  From
\eqref{eq-Cij} one has
\[
(C_s)_{i,j}=\bigl(s+1-\max(i,j-1)\bigr)_+,
\qquad 0\le i,j\le s+1.
\]
Let $L_n$ and $R_n$ be the unimodular difference matrices defined by
\[
(L_n)_{i,j}=\begin{cases}
1,& j=i,\\
-1,& j=i+1,\ 0\le i<n-1,\\
0,& \text{otherwise},
\end{cases}
\qquad
(R_n)_{i,j}=\begin{cases}
1,& i=j,\\
-1,& i=j+1,\ 0\le j<n-1,\\
0,& \text{otherwise}.
\end{cases}
\]
Thus left multiplication by $L_n$ takes adjacent row differences, while right
multiplication by $R_n$ takes adjacent column differences.  Both have
determinant $1$.

A direct substitution of the displayed formula for $(C_s)_{i,j}$ gives
\[
L_nM_s(y)R_n=
\begin{pmatrix}
2 & -W & 0 & \cdots & 0 & 0\\
-1 & 2 & -W & \ddots & \vdots & \vdots\\
0 & -1 & 2 & \ddots & 0 & 0\\
\vdots & \ddots & \ddots & \ddots & -W & 0\\
0 & \cdots & 0 & -1 & 2 & -W\\
0 & \cdots & 0 & 0 & -1 & 1
\end{pmatrix}.
\]
Taking the Schur complement of the final $1$ changes the preceding diagonal
entry from $2$ to $2-W=1-y$ and gives exactly $T_{s+1}(y)$.

We now track the perturbation.  Since $(B_s)_{i,j}=1$ if and only if
$i+j\le s$, and $0$ otherwise, the same row and column differences give
\[
L_nB_sR_n=
\begin{pmatrix}
E_s & 0\\
0 & 0
\end{pmatrix},
\]
where $E_s$ is the $(s+1)\times(s+1)$ matrix
\[
(E_s)_{p,q}=\begin{cases}
1,& p+q=s+2,\\
-1,& p+q=s+1,\\
0,& \text{otherwise},
\end{cases}
\qquad 1\le p,q\le s+1.
\]
Therefore
\[
\det\bigl(M_s(y)-tB_s\bigr)
=\det\bigl(T_{s+1}(y)-tE_s\bigr).
\]
By Lemma~\ref{lem-R-directional},
\[
\mathcal R_s(y)=\operatorname{Tr}\bigl(T_{s+1}(y)^*E_s\bigr).
\]
Since $E_s$ is supported on the two anti-diagonal strips $p+q=s+2$ and
$p+q=s+1$, with signs $+1$ and $-1$, respectively, this trace is precisely
\eqref{eq-R-antidiagonal-new}.
\end{proof}

\begin{lemma}[Adjugate entries of the reduced tridiagonal matrix]\label{lem-tridiagonal-adjugate}
Let
\[
\widehat P_n:=P_n(-y),\qquad \widehat Q_n:=Q_n(-y).
\]
For $1\le p,q\le s+1$, one has
\begin{equation}\label{eq-Apq-new}
A^{(s)}_{p,q}(y)
= W^{(q-p)_+}\,
\widehat P_{\min(p,q)}\,
\widehat Q_{s+2-\max(p,q)},
\end{equation}
where $(q-p)_+:=\max(q-p,0)$.
\end{lemma}

\begin{proof}
Let $\theta_j$ denote the determinant of the leading $j\times j$ principal
minor of $T_{s+1}(y)$, with $\theta_0=1$.  For $j\ge1$, these determinants
satisfy
\[
\theta_j=2\theta_{j-1}-W\theta_{j-2},
\qquad \theta_0=1,\quad \theta_1=2,
\]
as long as the last row of $T_{s+1}(y)$ is not involved.  Hence
\[
\theta_j=\widehat P_{j+1}\qquad (0\le j\le s).
\]
Similarly, let $\phi_j$ be the determinant of the trailing principal minor
using rows and columns $j,j+1,\ldots,s+1$, and set $\phi_{s+2}=1$.  Since
$\phi_{s+1}=1-y=\widehat Q_2$ and the same continuant recurrence runs upward,
we have
\[
\phi_j=\widehat Q_{s+3-j}\qquad (1\le j\le s+2).
\]

The standard tridiagonal adjugate formula now gives the result.  Indeed, if
$p\le q$, then the product of the superdiagonal entries from $p$ to $q-1$ is
$(-W)^{q-p}$, whose sign is cancelled by the cofactor sign, and therefore
\[
A^{(s)}_{p,q}(y)=W^{q-p}\theta_{p-1}\phi_{q+1}
=W^{q-p}\widehat P_p\widehat Q_{s+2-q}.
\]
If $p>q$, the product of the subdiagonal entries from $q$ to $p-1$ is
$(-1)^{p-q}$, again cancelled by the cofactor sign, giving
\[
A^{(s)}_{p,q}(y)=\theta_{q-1}\phi_{p+1}
=\widehat P_q\widehat Q_{s+2-p}.
\]
These two cases are exactly \eqref{eq-Apq-new}.
\end{proof}

\begin{lemma}[Two product identities]\label{lem-PQ-product-identities}
For every $r\ge1$,
\begin{equation}\label{eq-PQ-products-new}
2\widehat P_r\widehat Q_r=\widehat P_{2r},
\qquad
2\widehat P_r\widehat Q_{r+1}=\widehat P_{2r+1}-W^r.
\end{equation}
\end{lemma}

\begin{proof}
From
\[
(1+\sqrt{-y})^{a+b}=(1+\sqrt{-y})^a(1+\sqrt{-y})^b
\]
and comparison of the odd parts, we get
\[
\widehat P_{a+b}=\widehat P_a\widehat Q_b+\widehat Q_a\widehat P_b.
\]
Taking $a=b=r$ gives $2\widehat P_r\widehat Q_r=\widehat P_{2r}$.

For the second identity, first note that
\[
\widehat P_{r+1}\widehat Q_r-\widehat P_r\widehat Q_{r+1}=W^r.
\]
This follows either by induction from the first-order system in
Lemma~\ref{lem-PQ-rec1}, or directly from the closed form
$\widehat Q_n+\sqrt{-y}\,\widehat P_n=(1+\sqrt{-y})^n$.
Combining this with
\[
\widehat P_{2r+1}
=\widehat P_r\widehat Q_{r+1}+\widehat Q_r\widehat P_{r+1}
\]
yields
\[
\widehat P_{2r+1}=2\widehat P_r\widehat Q_{r+1}+W^r,
\]
which proves the second identity.
\end{proof}

\begin{lemma}[Adjugate trace for the odd cyclic series]\label{lem-Rs-explicit}
Let
\[
\mathcal{R}_s(y):=\operatorname{Tr}\bigl((I-yC_s)^*B_s\bigr).
\]
Put $W=1+y$ and $\widehat P_n=P_n(-y)$.  If $s=2\ell$, then
\begin{align}\label{eq-R-even}
\mathcal{R}_{2\ell}(y)
={}&\frac12\sum_{r=1}^{\ell}\bigl(1+W^{2\ell+2-2r}\bigr)\widehat P_{2r}
+\frac12\widehat P_{2\ell+2}\notag\\
&-\frac12\sum_{r=1}^{\ell}\bigl(1+W^{2\ell+1-2r}\bigr)
\bigl(\widehat P_{2r+1}-W^r\bigr).
\end{align}
If $s=2\ell+1$, then
\begin{align}\label{eq-R-odd}
\mathcal{R}_{2\ell+1}(y)
={}&\frac12\sum_{r=1}^{\ell+1}\bigl(1+W^{2\ell+3-2r}\bigr)\widehat P_{2r}\notag\\
&-\frac12\sum_{r=1}^{\ell}\bigl(1+W^{2\ell+2-2r}\bigr)
\bigl(\widehat P_{2r+1}-W^r\bigr)
-\frac12\bigl(\widehat P_{2\ell+3}-W^{\ell+1}\bigr).
\end{align}
\end{lemma}

\begin{proof}
By Lemma~\ref{lem-R-antidiagonal-reduction},
\[
\mathcal R_s(y)=\mathcal S^{(1)}_s(y)-\mathcal S^{(2)}_s(y),
\]
where
\[
\mathcal S^{(1)}_s(y):=\sum_{p+q=s+2}A^{(s)}_{p,q}(y),
\qquad
\mathcal S^{(2)}_s(y):=\sum_{p+q=s+1}A^{(s)}_{p,q}(y).
\]

We first evaluate $\mathcal S^{(1)}_s(y)$.  On the strip $p+q=s+2$, if
$r=\min(p,q)$, then Lemma~\ref{lem-tridiagonal-adjugate} gives
\[
A^{(s)}_{p,q}(y)=W^{(q-p)_+}\widehat P_r\widehat Q_r.
\]
Pairing the symmetric terms $(p,q)$ and $(q,p)$ yields
\[
\bigl(1+W^{s+2-2r}\bigr)\widehat P_r\widehat Q_r.
\]
Hence, by Lemma~\ref{lem-PQ-product-identities},
\[
\mathcal S^{(1)}_{2\ell}(y)
=\frac12\sum_{r=1}^{\ell}
\bigl(1+W^{2\ell+2-2r}\bigr)\widehat P_{2r}
+\frac12\widehat P_{2\ell+2},
\]
because the central term $p=q=\ell+1$ contributes
$\widehat P_{\ell+1}\widehat Q_{\ell+1}=\widehat P_{2\ell+2}/2$.  Similarly,
when $s=2\ell+1$, there is no central term on this strip, and therefore
\[
\mathcal S^{(1)}_{2\ell+1}(y)
=\frac12\sum_{r=1}^{\ell+1}
\bigl(1+W^{2\ell+3-2r}\bigr)\widehat P_{2r}.
\]

We next evaluate $\mathcal S^{(2)}_s(y)$.  On the strip $p+q=s+1$, if
$r=\min(p,q)$, then Lemma~\ref{lem-tridiagonal-adjugate} gives
\[
A^{(s)}_{p,q}(y)=W^{(q-p)_+}\widehat P_r\widehat Q_{r+1}.
\]
Pairing symmetric terms and using
\[
2\widehat P_r\widehat Q_{r+1}=\widehat P_{2r+1}-W^r
\]
from Lemma~\ref{lem-PQ-product-identities}, we obtain
\[
\mathcal S^{(2)}_{2\ell}(y)
=\frac12\sum_{r=1}^{\ell}
\bigl(1+W^{2\ell+1-2r}\bigr)
\bigl(\widehat P_{2r+1}-W^r\bigr).
\]
For $s=2\ell+1$, the strip $p+q=s+1=2\ell+2$ has the central term
$p=q=\ell+1$.  Thus
\[
\mathcal S^{(2)}_{2\ell+1}(y)
=\frac12\sum_{r=1}^{\ell}
\bigl(1+W^{2\ell+2-2r}\bigr)
\bigl(\widehat P_{2r+1}-W^r\bigr)
+\frac12\bigl(\widehat P_{2\ell+3}-W^{\ell+1}\bigr).
\]
Substituting these four expressions into
$\mathcal R_s(y)=\mathcal S^{(1)}_s(y)-\mathcal S^{(2)}_s(y)$ gives
\eqref{eq-R-even} and \eqref{eq-R-odd}.
\end{proof}

\begin{remark}\label{rem-Rs-small-check}
For instance, when $s=2$, Lemma~\ref{lem-Rs-explicit} gives
\[
\mathcal R_2(y)=2+y+2y^2,
\qquad
\Omega_2^{\mathrm{cyc}}(y)=\frac{13}{1-6y+y^2},
\]
in agreement with a direct trace computation from $C_2$ and $B_2$.
\end{remark}

\begin{proof}[Proof of Theorem~\ref{thm-cyclic-intro}]
For the even-dimensional series, Jacobi's formula gives
\[
\frac{d}{dy}\det(I-yC_s)
=-\det(I-yC_s)\operatorname{Tr}\bigl(C_s(I-yC_s)^{-1}\bigr).
\]
Together with Lemma~\ref{lem-det}, this yields
\[
\Psi_s^{\mathrm{cyc}}(y)
=-\frac{\frac{d}{dy}Q_{s+2}(-y)}{Q_{s+2}(-y)}.
\]
Lemma~\ref{lem-Qderiv} with $n=s+2$ proves \eqref{eq-Psicyc-intro}.

For the odd-dimensional series, use
\[
C_s(I-yC_s)^{-1}=\frac{(I-yC_s)^{-1}-I}{y}.
\]
Taking traces after multiplying by $B_s$ gives
\[
\Omega_s^{\mathrm{cyc}}(y)
=\frac{1}{y}\left(
\operatorname{Tr}\bigl((I-yC_s)^{-1}B_s\bigr)-\operatorname{Tr}(B_s)
\right).
\]
Since $\operatorname{Tr}(B_s)=\lfloor s/2\rfloor+1$ and
$(I-yC_s)^{-1}=M_s^*/Q_{s+2}(-y)$, Lemma~\ref{lem-Rs-explicit} proves
\eqref{eq-Omegacyc-intro}.

It remains to record the asymptotics.  Let
$\alpha_s=\pi/(2(s+2))$ and $y_s^*=\tan^2\alpha_s$.  The denominator
$Q_{s+2}(-y)$ has a simple dominant zero at $y_s^*$.  As in
Remark~\ref{rem-euler},
\begin{equation}\label{eq-Qprime-cyc}
\left.\frac{d}{dy}Q_{s+2}(-y)\right|_{y=y_s^*}
=-\frac{s+2}{2\sin\alpha_s\cos^{s-1}\alpha_s}.
\end{equation}
For the even-dimensional series, the numerator in \eqref{eq-Psicyc-intro}
evaluates at $y_s^*$ to the negative of \eqref{eq-Qprime-cyc}.  Therefore
the residue computation gives the leading constant $1$, namely
$N_{2m}^{\mathrm{cyc}}(s)\sim (y_s^*)^{-m}$.
For the odd-dimensional series, the numerator of \eqref{eq-Omegacyc-intro}
at $y_s^*$ is $\mathcal{R}_s(y_s^*)/y_s^*$, because $Q_{s+2}(-y_s^*)=0$.  Extracting the
coefficient of $y^{m-1}$ and using \eqref{eq-Qprime-cyc} gives
\[
N_{2m+1}^{\mathrm{cyc}}(s)
\sim -\frac{\mathcal{R}_s(y_s^*)}{y_s^*\left.\frac{d}{dy}Q_{s+2}(-y)\right|_{y=y_s^*}}
(y_s^*)^{-m}
=\xi_s(y_s^*)^{-m},
\]
where $\xi_s$ is the constant stated in Theorem~\ref{thm-cyclic-intro}.
\end{proof}

\medskip
\noindent\textit{Recurrences and trigonometric forms.}

\begin{proof}[Proof of Theorem~\ref{thm-cyclic-recurrence}]
By Theorem~\ref{thm-cyclic-intro},
\[
\widehat\Psi_s^{\mathrm{cyc}}(y)=\frac{P_{s+1}(-y)}{Q_{s+2}(-y)}\qquad(s\ge1),
\]
and the auxiliary initial value is
$\widehat\Psi_0^{\mathrm{cyc}}(y)=P_1(-y)/Q_2(-y)=1/(1-y)$.  Assume $s\ge1$ and write
$\widehat\Psi_{s-1}^{\mathrm{cyc}}(y)=P_s(-y)/Q_{s+1}(-y)$.  From the first-order system
\eqref{eq-P-rec1}--\eqref{eq-Q-rec1},
\[
P_{s+1}(-y)=P_s(-y)+Q_s(-y),\qquad
Q_{s+1}(-y)=Q_s(-y)-yP_s(-y).
\]
Thus $Q_s(-y)=Q_{s+1}(-y)+yP_s(-y)$, and hence
\[
P_{s+1}(-y)=Q_{s+1}(-y)+(1+y)P_s(-y)
=Q_{s+1}(-y)\bigl(1+(1+y)\widehat\Psi_{s-1}^{\mathrm{cyc}}(y)\bigr).
\]
Similarly,
\[
Q_{s+2}(-y)=Q_{s+1}(-y)-yP_{s+1}(-y)
=Q_{s+1}(-y)\bigl(1-y-y(1+y)\widehat\Psi_{s-1}^{\mathrm{cyc}}(y)\bigr).
\]
Dividing the two identities gives \eqref{eq-Lcyc-rec-main}.  Since
$\Psi_s^{\mathrm{cyc}}(y)=\frac{s+2}{2}\widehat\Psi_s^{\mathrm{cyc}}(y)$ and
$\widehat\Psi_{s-1}^{\mathrm{cyc}}(y)=\frac{2}{s+1}\Psi_{s-1}^{\mathrm{cyc}}(y)$,
substitution into \eqref{eq-Lcyc-rec-main} gives the equivalent recurrence
\eqref{eq-Psicyc-rec-main} for the unnormalized series.

For the coefficient recurrence, write a cyclic generating function in the form
\[
\sum_{m\ge1}a_m y^{m-1}=\frac{A(y)}{Q_{s+2}(-y)},
\qquad
Q_{s+2}(-y)=\sum_{j=0}^{\nu_s}(-1)^j\binom{s+2}{2j}y^j,
\quad \nu_s=\left\lfloor\frac{s+2}{2}\right\rfloor.
\]
Multiplying by the denominator and comparing coefficients of $y^{m-1}$ for
$m-1>\deg A$ gives \eqref{eq-cyclic-coeff-rec-main}.  For the even cyclic
series the numerator has degree at most $\nu_s-1$, so the recurrence starts at
$m\ge \nu_s+1$.  For the odd cyclic series, the numerator is the one appearing
in Theorem~\ref{thm-cyclic-intro}; hence the same recurrence holds after the
finite set of initial terms determined by that numerator.
\end{proof}

\begin{proof}[Proof of Theorem~\ref{thm-cyclic-trig}]
Set $\theta=\arctan\sqrt y$, so that $y=\tan^2\theta$.  By
\eqref{eq-QP-trig},
\[
P_n(-y)=\frac{\sin(n\theta)}{\sin\theta\cos^{n-1}\theta},
\qquad
Q_n(-y)=\frac{\cos(n\theta)}{\cos^n\theta}.
\]
Substituting these formulas into
\eqref{eq-Psicyc-intro} immediately yields \eqref{eq-Psicyc-trig-main}.
For the odd series, substitute
\[
Q_{s+2}(-\tan^2\theta)=\frac{\cos((s+2)\theta)}{\cos^{s+2}\theta}
\]
into \eqref{eq-Omegacyc-intro}.  After multiplying numerator and denominator by
$\cos^{s+2}\theta$, we obtain \eqref{eq-Omegacyc-trig-main}.
\end{proof}

\begin{example}[Cyclic generating functions in small cases]\label{ex-cyclic}
\textit{Case $s=1$.}\;
\[
\Psi_1^{\mathrm{cyc}}(y)=\frac{3}{1-3y},
\qquad
\Omega_1^{\mathrm{cyc}}(y)=\frac{5}{1-3y}.
\]
Thus $N_{2m}^{\mathrm{cyc}}(1)=3^m$ and
$N_{2m+1}^{\mathrm{cyc}}(1)=5\cdot3^{m-1}$ for $m\ge1$.

\medskip\noindent\textit{Case $s=2$.}\;
\[
\Psi_2^{\mathrm{cyc}}(y)=\frac{6-2y}{1-6y+y^2},
\qquad
\Omega_2^{\mathrm{cyc}}(y)=\frac{13}{1-6y+y^2}.
\]
Thus $N_2^{\mathrm{cyc}}(2)=6$, $N_4^{\mathrm{cyc}}(2)=34$,
$N_6^{\mathrm{cyc}}(2)=198$, while
$N_3^{\mathrm{cyc}}(2)=13$, $N_5^{\mathrm{cyc}}(2)=78$.

\medskip\noindent\textit{Case $s=3$.}\;
\[
\Psi_3^{\mathrm{cyc}}(y)=\frac{10-10y}{1-10y+5y^2},
\qquad
\Omega_3^{\mathrm{cyc}}(y)=\frac{27-8y+2y^2}{1-10y+5y^2}.
\]
In particular $N_3^{\mathrm{cyc}}(3)=27$, $N_5^{\mathrm{cyc}}(3)=262$, and
$N_7^{\mathrm{cyc}}(3)=2487$.
\end{example}

\section{Ehrhart series and \texorpdfstring{$h^*$}{h*}-vectors}\label{sec-3}

The argument in this section proceeds in four stages, paralleling the
structure of Section~\ref{sec-2}. First (\S\ref{subsec-ehrhart-setup}), we
set up the Ehrhart-theoretic framework: we record the standard inversion
formula and prove the general $h^*$-extraction theorem
(Theorem~\ref{thm-hstar-extraction}). Second
(\S\ref{subsec-fixedm-transfer}), we introduce the fixed-dilation transfer
matrices and reduce the resolvent computation to a compressed matrix of size
$(sm+1)\times(sm+1)$. Third (\S\ref{subsec-eval}), we evaluate the resolvent
via a triangular action on explicit test vectors and thereby prove the general
rational formula Theorem~\ref{thm-Gm-general}. Fourth
(\S\ref{subsec-consequences}), we assemble the downstream consequences:
parity recurrences, volume generating functions, the bivariate master
identity, and low-degree $h^*$-coefficient formulas. The section closes with
two dedicated subsections: \S\ref{subsec-s1} specializes to $s=1$ to verify
the formulas explicitly and prove Theorem~\ref{thm-gorenstein-s1}, while
\S\ref{subsec-gorenstein} records the resulting non-Gorenstein behavior for
$s\ge2$.

\subsection{Ehrhart setup and the extraction theorem}
\label{subsec-ehrhart-setup}

Recall that, for $d\ge2$,
$L_d^{(s)}(m) := \#(m\cP_d^{(s)} \cap \Z^d)$ denotes the Ehrhart counting
function of $\cP_d^{(s)}$.  It is an Ehrhart polynomial of degree $d$, and its
generating function
\begin{equation}\label{eq-Ehr-ds}
\Ehr_{\cP_d^{(s)}}(z)
= 1 + \sum_{m \ge 1} L_d^{(s)}(m)\, z^m
= \frac{h_d^{(s)}(z)}{(1-z)^{d+1}}
\end{equation}
defines the $h^*$-polynomial
$h_d^{(s)}(z)=\sum_{k=0}^d h_{d,k}^{(s)}z^k$.  Thus, as in the
introduction, $h_{d,k}^{(s)}$ denotes the $k$-th Ehrhart $h^*$-coefficient.

The following inversion formula is standard; see \cite[Chapter~3]{beck}.

\begin{lemma}[Inversion formula]\label{lem-hstar-inversion}
For fixed $d \ge 2$, $s \ge 1$, and $k \ge 0$, with the convention
$h_{d,k}^{(s)}=0$ for $k>d$, one has
\begin{equation}\label{eq-hstar-inversion}
h_{d,k}^{(s)}
= \sum_{i=0}^{k} (-1)^i \binom{d+1}{i} L_d^{(s)}(k-i).
\end{equation}
\end{lemma}

The extraction theorem expresses the coefficient-generating functions
$H_k^{(s)}(y)$ directly in terms of the fixed-dilation series $G_m^{(s)}(y)$.
We therefore extend the family to dimensions $d=0$ and $d=1$ so that $G_m^{(s)}(y)$ is
well-defined for all $d\ge0$. Set $\cP_0^{(s)}:=\{0\}$ (so $L_0^{(s)}(m)=1$),
and $\cP_1^{(s)}:=[0,s]$ motivated by the projection of $\cP_2^{(s)}$ onto the
first coordinate: since $x_1+x_2\le s$ and $x_2\ge0$ imply $x_1\le s$, this
gives $L_1^{(s)}(m)=ms+1$. These low-dimensional conventions are used only
within this section. Define
\[
G_m^{(s)}(y) := \sum_{d \ge 0} L_d^{(s)}(m)\, y^d,
\qquad
H_k^{(s)}(y) := \sum_{d \ge 0} h_{d,k}^{(s)}\, y^d.
\]

\begin{proof}[Proof of Theorem~\ref{thm-hstar-extraction}]
For each fixed $i\ge 0$, differentiating
\[
y\,G_{k-i}^{(s)}(y)=\sum_{d\ge 0}L_d^{(s)}(k-i)\,y^{d+1}
\]
termwise gives
\[
\frac{1}{i!}\,y^{i-1}\frac{d^i}{dy^i}\!\Bigl(y\,G_{k-i}^{(s)}(y)\Bigr)
=\sum_{d\ge 0}\binom{d+1}{i}L_d^{(s)}(k-i)\,y^d.
\]
(For $i\ge1$ the factor $y^{i-1}$ is cancelled by $y^i$ produced when
differentiating $y\,G_{k-i}^{(s)}(y)$, so each summand is a formal power
series in $y$ with no negative-exponent terms.)
Therefore the right-hand side of \eqref{eq-Hk-extraction} has coefficient
of $y^d$ equal to
\[
\sum_{i=0}^k(-1)^i\binom{d+1}{i}L_d^{(s)}(k-i) = h_{d,k}^{(s)}
\quad\text{for }d\ge 2,
\]
by Lemma~\ref{lem-hstar-inversion}.

For $d=0$: if $k=0$, the coefficient is $L_0^{(s)}(0)=1=h_{0,0}^{(s)}$; if $k\ge1$,
only $i=0,1$ contribute and $L_0^{(s)}\equiv1$ gives $1-1=0=h_{0,k}^{(s)}$.
For $d=1$: using $L_1^{(s)}(m)=sm+1$, we get $h_{1,0}^{(s)}=1$,
$h_{1,1}^{(s)}=(s+1)-2=s-1$, and $h_{1,k}^{(s)}=0$ for $k\ge2$ (since
$L_1^{(s)}$ is linear). These match the Ehrhart series
$\Ehr_{\cP_1^{(s)}}(z) = (1+(s-1)z)/(1-z)^2$.
\end{proof}

Theorem~\ref{thm-hstar-extraction} is purely formal until the series
$G_m^{(s)}(y)$ are made explicit. The next two subsections supply these
explicit formulas for all $s,m\ge1$.

\subsection{Fixed-\texorpdfstring{$m$}{m} transfer matrices and resolvent
compression}
\label{subsec-fixedm-transfer}

Just as Section~\ref{sec-2} encoded the dimension-varying problem via the
double-step matrix $C_s$, we now encode the fixed-dilation problem via an
analogous pair of transition matrices indexed by integer vectors of length up
to $(s+1)m$.

Fix integers $s,m\ge1$. Define matrices of size
$((s+1)m+1)\times((s+1)m+1)$ by
\[
(\mathsf X_m^{(s)})_{ij}=\mathbf{1}[i+j\le sm],
\qquad
(\mathsf Y_m^{(s)})_{ij}=\mathbf{1}[i+j\le (s+1)m],
\]
for $0\le i,j\le (s+1)m$,
set $\mathsf Z_m^{(s)}:=\mathsf X_m^{(s)}\mathsf Y_m^{(s)}$ and
$\mathbf e:=(1,\dots,1)^\top\in\R^{(s+1)m+1}$, and introduce the parity-split
dimension-generating functions for the dilates:
\[
\mathcal O_m^{(s)}(z):=\mathbf e^\top(I-z\mathsf Z_m^{(s)})^{-1}\mathbf e,
\qquad
\mathcal E_m^{(s)}(z):=\mathbf e^\top(I-z\mathsf Z_m^{(s)})^{-1}\mathsf X_m^{(s)}\mathbf e.
\]
Here $\mathcal O_m^{(s)}(z)$ collects the odd-dimensional counts and
$\mathcal E_m^{(s)}(z)$ the even-dimensional counts, in exact parallel with
$\Phi_s$ and $\Psi_s$ from Section~\ref{sec-2}.
For clarity, $\Delta_{s,m}$ below is the common denominator of these two
series, while the vectors $U_j$ and $V_j$ are auxiliary test vectors used only
to evaluate the compressed resolvent in Lemma~\ref{lem-Gm-triangular}.

For the explicit rational formulas, define, for $0\le j\le s$,
\[
a_j^{(s,m)}:=\binom{(s+1-j)m+j+1}{2j},
\qquad
b_j^{(s,m)}:=\binom{(s+1-j)m+j+1}{2j+1},
\]
and set $\Delta_{s,m}(z):=\sum_{j=0}^{s}(-1)^j a_j^{(s,m)}z^j$. This is the
common denominator of $\mathcal E_m^{(s)}$ and $\mathcal O_m^{(s)}$; note
that $a_0^{(s,m)}=1$, so $\Delta_{s,m}(0)=1$.

\begin{remark}\label{rem-delta-degenerate}
The denominator $\Delta_{s,m}(z)$ has degree at most $s$, but the actual
degree is
\[
\deg\Delta_{s,m}
= \max\bigl\{j\in\{0,\ldots,s\}: a_j^{(s,m)}\ne0\bigr\}
= s - \Bigl\lfloor\frac{s-1}{m+1}\Bigr\rfloor
\]
for $m\ge1$, since $a_j^{(s,m)}=\binom{(s+1-j)m+j+1}{2j}\ne0$ iff
$(s+1-j)m+j+1\ge 2j$, equivalently
\[
j\le \left\lfloor\frac{(s+1)m+1}{m+1}\right\rfloor
= s-\left\lfloor\frac{s-1}{m+1}\right\rfloor .
\]
In particular $\deg\Delta_{s,m}=s$ whenever $m\ge s-2$, and the degree
first drops below $s$ for $m=1$ and $s\ge3$.
For example, $\Delta_{3,1}(z)=1-10z+5z^2$ has degree $2<3=s$.
\end{remark}

The key to the proof of Theorem~\ref{thm-Gm-general} is to compress the
large matrix $\mathsf Z_m^{(s)}$ to a smaller one via the push-through identity.
The following lemma carries out this compression and expresses both
$\mathcal E_m^{(s)}$ and $\mathcal O_m^{(s)}$ in terms of a matrix
$\mathsf A$ of size $(sm+1)\times(sm+1)$---a reduction exactly parallel to
the passage from $C_s$ to its resolvent in Section~\ref{sec-2}.

\begin{lemma}[Compressed resolvent formulas]\label{lem-Gm-compression}
Let $\mathsf S\in\R^{((s+1)m+1)\times(sm+1)}$ have upper block $J_{sm}$
($(J_{sm})_{ab}=\mathbf{1}[a+b\le sm]$) and lower block~$0$, and let
$\mathsf T=(I_{sm+1}\ 0)\in\R^{(sm+1)\times((s+1)m+1)}$.
Set $\mathsf R:=\mathsf T\mathsf Y_m^{(s)}$ and $\mathsf A:=\mathsf R\mathsf S$.
Then $\mathsf X_m^{(s)}=\mathsf S\mathsf T$, $\mathsf Z_m^{(s)}=\mathsf S\mathsf R$, and
\begin{equation}\label{eq-OE-general-compressed}
\mathcal E_m^{(s)}(z)=\mathbf c^\top(I-z\mathsf A)^{-1}\mathbf 1,
\qquad
\mathcal O_m^{(s)}(z)=(s+1)m+1+z\,\mathbf c^\top(I-z\mathsf A)^{-1}\mathbf d,
\end{equation}
where $\mathbf c^\top:=\mathbf e^\top\mathsf S=(sm+1,sm,\ldots,1)$ and
$\mathbf d:=\mathsf R\mathbf e=((s+1)m+1,\ldots,m+1)^\top$. The matrix
$\mathsf A$ has entries
\begin{equation}\label{eq-Aab-general}
\mathsf A_{ab}=sm-b+1-(a-b-m)_+,\qquad 0\le a,b\le sm.
\end{equation}
\end{lemma}

\begin{proof}
We first explain the two factorizations.  The matrix $\mathsf X_m^{(s)}$ has
entries
\[
(\mathsf X_m^{(s)})_{ij}=\mathbf 1[i+j\le sm],
\qquad 0\le i,j\le (s+1)m .
\]
If either $i>sm$ or $j>sm$, then $i+j>sm$, so the corresponding entry is zero.
Hence $\mathsf X_m^{(s)}$ has only one nonzero block, namely its upper-left
$(sm+1)\times(sm+1)$ block $J_{sm}$.  This is exactly the factorization
\[
\mathsf X_m^{(s)}=\mathsf S\mathsf T.
\]
Therefore
\[
\mathsf Z_m^{(s)}
=\mathsf X_m^{(s)}\mathsf Y_m^{(s)}
=\mathsf S\mathsf T\mathsf Y_m^{(s)}
=\mathsf S\mathsf R,
\]
where $\mathsf R=\mathsf T\mathsf Y_m^{(s)}$.  By definition
$\mathsf A=\mathsf R\mathsf S$.

We now record explicitly the push-through identity used in the compression.
For rectangular matrices $S$ and $R$ of compatible sizes, one has the formal
power-series identity
\begin{equation}\label{eq-push-through-fixedm}
(I-zSR)^{-1}=I+zS(I-zRS)^{-1}R.
\end{equation}
Indeed, setting $B=(I-zRS)^{-1}$, we compute
\begin{align*}
(I-zSR)(I+zSBR)
&=I+zSBR-zSR-z^2SRSBR \\
&=I+zS\bigl(B-I-zRSB\bigr)R.
\end{align*}
Since $B=(I-zRS)^{-1}$ implies $B-I=zRSB$, the last term is zero, proving
\eqref{eq-push-through-fixedm}.  Multiplying \eqref{eq-push-through-fixedm}
on the right by $S$ also gives the companion identity
\begin{equation}\label{eq-push-through-companion-fixedm}
(I-zSR)^{-1}S=S(I-zRS)^{-1}.
\end{equation}
Applying \eqref{eq-push-through-fixedm} and
\eqref{eq-push-through-companion-fixedm} with $S=\mathsf S$ and
$R=\mathsf R$, we get
\begin{equation}\label{eq-resolvent-compressed-fixedm}
(I-z\mathsf Z_m^{(s)})^{-1}
=I+z\mathsf S(I-z\mathsf A)^{-1}\mathsf R,
\end{equation}
and
\begin{equation}\label{eq-resolvent-compressed-S-fixedm}
(I-z\mathsf Z_m^{(s)})^{-1}\mathsf S
=\mathsf S(I-z\mathsf A)^{-1}.
\end{equation}

We first compute the even-dimensional series.  Since
$\mathsf X_m^{(s)}=\mathsf S\mathsf T$, we have
\begin{align*}
\mathcal E_m^{(s)}(z)
&=\mathbf e^\top(I-z\mathsf Z_m^{(s)})^{-1}\mathsf X_m^{(s)}\mathbf e \\
&=\mathbf e^\top(I-z\mathsf Z_m^{(s)})^{-1}\mathsf S\mathsf T\mathbf e.
\end{align*}
Using \eqref{eq-resolvent-compressed-S-fixedm}, this becomes
\[
\mathcal E_m^{(s)}(z)
=\mathbf e^\top\mathsf S(I-z\mathsf A)^{-1}\mathsf T\mathbf e.
\]
Now $\mathbf e^\top\mathsf S=\mathbf c^\top$ and
$\mathsf T\mathbf e=\mathbf 1$, so
\[
\mathcal E_m^{(s)}(z)=\mathbf c^\top(I-z\mathsf A)^{-1}\mathbf 1.
\]

Next, for the odd-dimensional series, \eqref{eq-resolvent-compressed-fixedm}
gives
\begin{align*}
\mathcal O_m^{(s)}(z)
&=\mathbf e^\top(I-z\mathsf Z_m^{(s)})^{-1}\mathbf e \\
&=\mathbf e^\top\mathbf e
+z\,\mathbf e^\top\mathsf S(I-z\mathsf A)^{-1}\mathsf R\mathbf e.
\end{align*}
Since $\mathbf e^\top\mathbf e=(s+1)m+1$,
$\mathbf e^\top\mathsf S=\mathbf c^\top$, and
$\mathsf R\mathbf e=\mathbf d$, we obtain
\[
\mathcal O_m^{(s)}(z)
=(s+1)m+1+z\,\mathbf c^\top(I-z\mathsf A)^{-1}\mathbf d.
\]
This proves \eqref{eq-OE-general-compressed}.

It remains to compute the vectors $\mathbf c$, $\mathbf d$, and the entries of
$\mathsf A$.  For $0\le b\le sm$, the $b$-th component of
$\mathbf c^\top=\mathbf e^\top\mathsf S$ is the column sum of $\mathsf S$:
\[
c_b=\#\{0\le i\le sm:i+b\le sm\}=sm-b+1.
\]
Thus
\[
\mathbf c^\top=(sm+1,sm,\ldots,1).
\]
Similarly, because $\mathsf R=\mathsf T\mathsf Y_m^{(s)}$ consists of the
first $sm+1$ rows of $\mathsf Y_m^{(s)}$, for $0\le a\le sm$ we have
\[
d_a=(\mathsf R\mathbf e)_a
=\#\{0\le j\le (s+1)m:a+j\le (s+1)m\}
=(s+1)m-a+1.
\]
Therefore
\[
\mathbf d=((s+1)m+1,(s+1)m,\ldots,m+1)^\top.
\]
Finally,
\begin{align*}
\mathsf A_{ab}
&=(\mathsf R\mathsf S)_{ab}
=\sum_{t=0}^{(s+1)m}\mathsf R_{a,t}\mathsf S_{t,b} \\
&=\#\{t:0\le t\le sm-b,\ a+t\le (s+1)m\} \\
&=\min\{sm-b,(s+1)m-a\}+1.
\end{align*}
Since
\[
\min\{sm-b,(s+1)m-a\}=sm-b-(a-b-m)_+,
\]
we get
\[
\mathsf A_{ab}=sm-b+1-(a-b-m)_+,
\]
which proves \eqref{eq-Aab-general} and completes the proof.
\end{proof}

\subsection{Triangular evaluation and proof of Theorem~\ref{thm-Gm-general}}
\label{subsec-eval}

Lemma~\ref{lem-Gm-compression} reduces the problem to evaluating
$\mathbf c^\top(I-z\mathsf A)^{-1}\mathbf v$ for specific test vectors $\mathbf v$.
The key observation---parallel to the role of $Q_n(-y)$ and $P_n(-y)$ in
evaluating the adjugate numerators in Section~\ref{sec-2}---is that $\mathsf A$
acts in a triangular, rank-one-shift fashion on two explicit families of
vectors $U_j$ and $V_j$. This allows backward induction to compute the
resolvent action completely.

\begin{lemma}[Triangular action on test vectors]\label{lem-Gm-triangular}
For $0\le j\le s$, define vectors $U_j,V_j\in\R^{sm+1}$ by
\[
U_j(a):=\binom{(a-jm)_++j}{2j},
\qquad
V_j(a):=\binom{(a-jm)_++j+1}{2j+1},
\qquad 0\le a\le sm.
\]
Then $U_s\equiv0$ and $V_s\equiv0$. For $0\le j\le s-1$,
\begin{equation}\label{eq-AU-general}
\mathsf A U_j = a_{j+1}^{(s,m)}\mathbf 1 - U_{j+1},
\qquad
\mathbf c^\top U_j = a_{j+1}^{(s,m)},
\end{equation}
and
\begin{equation}\label{eq-AV-general}
\mathsf A V_j = \widetilde b_{j+1}^{(s,m)}\mathbf 1 - V_{j+1},
\qquad
\mathbf c^\top V_j = \widetilde b_{j+1}^{(s,m)},
\end{equation}
where $\widetilde b_{j+1}^{(s,m)}:=\binom{(s-j)m+j+3}{2j+3}=a_{j+1}^{(s,m)}+b_{j+1}^{(s,m)}$.
\end{lemma}

\begin{proof}
Since $a\le sm$, the quantity $(a-sm)_+$ vanishes. Therefore
\[
U_s(a)=\binom{s}{2s}=0,
\qquad
V_s(a)=\binom{s+1}{2s+1}=0
\]
for every $0\le a\le sm$. Hence $U_s\equiv0$ and $V_s\equiv0$.

We prove the identities for $U_j$ first.  Fix $0\le j\le s-1$.  From
\eqref{eq-Aab-general},
\begin{align*}
(\mathsf AU_j)(a)
&=\sum_{b=0}^{sm}
\bigl(sm-b+1-(a-b-m)_+\bigr)
\binom{(b-jm)_++j}{2j}.
\end{align*}
The summand is zero unless $b\ge jm$, because otherwise
$(b-jm)_+=0$ and, for $j\ge1$, $\binom{j}{2j}=0$; the case $j=0$ is also
covered since the lower limit is then $0$.  Put $t=b-jm$.  The contribution
from the first term $sm-b+1$ is
\[
\sum_{t=0}^{(s-j)m}\bigl((s-j)m-t+1\bigr)\binom{t+j}{2j}.
\]
By the hockey-stick identity
\[
\sum_{t=0}^{N}(N-t+1)\binom{t+j}{2j}
=\binom{N+j+2}{2j+2},
\]
with $N=(s-j)m$, this full contribution equals
\[
\binom{(s-j)m+j+2}{2j+2}=a_{j+1}^{(s,m)}.
\]

It remains to subtract the tail coming from $(a-b-m)_+$.  This term is nonzero
precisely when $b<a-m$, equivalently $t<a-(j+1)m$.  If $a\le (j+1)m$, there
is no tail, and $U_{j+1}(a)=0$, so the desired identity already holds.  If
$a>(j+1)m$, set $R=a-(j+1)m$.  Then the tail is
\[
\sum_{t=0}^{R-1}(R-t)\binom{t+j}{2j}.
\]
Using the second hockey-stick identity
\[
\sum_{t=0}^{R-1}(R-t)\binom{t+j}{2j}
=\binom{R+j+1}{2j+2},
\]
which, like the previous identity, follows by summing the ordinary
hockey-stick identity once more (see, e.g., \cite[Section~5.1]{GKP}), we obtain
\[
\binom{a-(j+1)m+j+1}{2j+2}=U_{j+1}(a).
\]
Combining the full contribution and the tail gives
\[
\mathsf AU_j=a_{j+1}^{(s,m)}\mathbf 1-U_{j+1}.
\]
Finally, when $a=0$, the row $\mathsf A_{0,b}$ is $sm-b+1=c_b$, because
$(0-b-m)_+=0$.  Thus the $a=0$ component of the identity just proved is
\[
\mathbf c^\top U_j=a_{j+1}^{(s,m)}-U_{j+1}(0).
\]
Since $U_{j+1}(0)=0$, this gives $\mathbf c^\top U_j=a_{j+1}^{(s,m)}$.

The proof for $V_j$ is parallel.  Starting from
\[
(\mathsf AV_j)(a)=\sum_{b=0}^{sm}
\bigl(sm-b+1-(a-b-m)_+\bigr)
\binom{(b-jm)_++j+1}{2j+1},
\]
the full contribution after the same change of variables is
\[
\sum_{t=0}^{(s-j)m}\bigl((s-j)m-t+1\bigr)\binom{t+j+1}{2j+1}
=\binom{(s-j)m+j+3}{2j+3}
=\widetilde b_{j+1}^{(s,m)}.
\]
The possible tail is, with $R=a-(j+1)m$,
\[
\sum_{t=0}^{R-1}(R-t)\binom{t+j+1}{2j+1}
=\binom{R+j+2}{2j+3}=V_{j+1}(a),
\]
and it is absent when $R\le0$, exactly when $V_{j+1}(a)=0$.  Therefore
\[
\mathsf AV_j=\widetilde b_{j+1}^{(s,m)}\mathbf 1-V_{j+1}.
\]
Taking again the $a=0$ component gives
\[
\mathbf c^\top V_j=\widetilde b_{j+1}^{(s,m)},
\]
because $V_{j+1}(0)=0$.  Finally,
\[
\widetilde b_{j+1}^{(s,m)}
=\binom{(s-j)m+j+3}{2j+3}
=\binom{(s-j)m+j+2}{2j+2}+
  \binom{(s-j)m+j+2}{2j+3}
=a_{j+1}^{(s,m)}+b_{j+1}^{(s,m)}
\]
by Pascal's rule.  This proves all assertions.
\end{proof}

With the triangular action established, the proof of
Theorem~\ref{thm-Gm-general} reduces to a backward induction that exactly
mirrors the induction used to identify $\det(I-yC_s)=Q_{s+2}(-y)$ in
Section~\ref{sec-2}.

\begin{proof}[Proof of Theorem~\ref{thm-Gm-general}]
The proof uses Lemma~\ref{lem-Gm-compression} to reduce the large resolvent to
the compressed matrix $\mathsf A$, and then uses Lemma~\ref{lem-Gm-triangular}
to evaluate the compressed resolvent on the test vectors $U_j$ and $V_j$.
By Lemma~\ref{lem-Gm-compression}, and using $\mathbf 1=U_0$, we have
\[
\mathcal E_m^{(s)}(z)=\mathbf c^\top(I-z\mathsf A)^{-1}U_0.
\]
Also $V_0(a)=a+1$, whereas $d_a=(s+1)m-a+1$; hence
\[
\mathbf d=((s+1)m+2)\mathbf 1-V_0=(b_0^{(s,m)}+1)U_0-V_0,
\]
because $b_0^{(s,m)}=(s+1)m+1$.  Therefore
\[
\mathcal O_m^{(s)}(z)
=b_0^{(s,m)}+z\,\mathbf c^\top(I-z\mathsf A)^{-1}
\bigl((b_0^{(s,m)}+1)U_0-V_0\bigr).
\]
For $0\le j\le s$, define
\[
T_j(z):=\mathbf c^\top(I-z\mathsf A)^{-1}U_j,
\qquad
S_j(z):=\mathbf c^\top(I-z\mathsf A)^{-1}V_j.
\]
Since $U_s\equiv V_s\equiv0$ by Lemma~\ref{lem-Gm-triangular}, we have
$T_s(z)=S_s(z)=0$.

\textit{Even part.}  We derive a backward recurrence for $T_j$.  Since
\[
(I-z\mathsf A)^{-1}=I+z(I-z\mathsf A)^{-1}\mathsf A,
\]
Lemma~\ref{lem-Gm-triangular} gives, for $0\le j\le s-1$,
\begin{align*}
T_j(z)
&=\mathbf c^\top U_j+z\,\mathbf c^\top(I-z\mathsf A)^{-1}\mathsf AU_j \\
&=a_{j+1}^{(s,m)}+z\,\mathbf c^\top(I-z\mathsf A)^{-1}
\bigl(a_{j+1}^{(s,m)}\mathbf 1-U_{j+1}\bigr) \\
&=a_{j+1}^{(s,m)}+z\bigl(a_{j+1}^{(s,m)}T_0(z)-T_{j+1}(z)\bigr).
\end{align*}
Set
\[
p_j(z):=\sum_{r=0}^{s-j-1}(-1)^r a_{j+r+1}^{(s,m)}z^r,
\qquad 0\le j\le s,
\]
where the empty sum is $p_s(z)=0$.  Then
\[
p_j(z)=a_{j+1}^{(s,m)}-zp_{j+1}(z).
\]
We claim by backward induction that
\[
T_j(z)=p_j(z)(1+zT_0(z))
\qquad (0\le j\le s).
\]
The claim is true for $j=s$, since both sides are zero.  If it holds for
$j+1$, then the recurrence for $T_j$ gives
\begin{align*}
T_j(z)
&=a_{j+1}^{(s,m)}+z\bigl(a_{j+1}^{(s,m)}T_0(z)-p_{j+1}(z)(1+zT_0(z))\bigr)\\
&=(a_{j+1}^{(s,m)}-zp_{j+1}(z))(1+zT_0(z))
=p_j(z)(1+zT_0(z)).
\end{align*}
Putting $j=0$ gives
\[
T_0(z)=p_0(z)(1+zT_0(z)).
\]
Thus
\[
T_0(z)=\frac{p_0(z)}{1-zp_0(z)}.
\]
Since
\[
1-zp_0(z)=1+\sum_{j=1}^{s}(-1)^j a_j^{(s,m)}z^j
=\Delta_{s,m}(z),
\]
and
\[
p_0(z)=\sum_{j=0}^{s-1}(-1)^j a_{j+1}^{(s,m)}z^j,
\]
we obtain \eqref{eq-Em-general}.

\textit{Odd part.}  The same argument, now using the $V_j$-identities in
Lemma~\ref{lem-Gm-triangular}, gives
\[
S_j(z)=q_j(z)(1+zT_0(z)),
\]
where
\[
q_j(z):=\sum_{r=0}^{s-j-1}(-1)^r\widetilde b_{j+r+1}^{(s,m)}z^r.
\]
Because $1+zT_0(z)=1/\Delta_{s,m}(z)$ from the even part, we have
\[
S_0(z)=\frac{q_0(z)}{\Delta_{s,m}(z)}.
\]
Therefore
\begin{align*}
\mathcal O_m^{(s)}(z)
&=b_0^{(s,m)}+z\bigl((b_0^{(s,m)}+1)T_0(z)-S_0(z)\bigr)\\
&=\frac{b_0^{(s,m)}\Delta_{s,m}(z)
+z\bigl((b_0^{(s,m)}+1)p_0(z)-q_0(z)\bigr)}{\Delta_{s,m}(z)}.
\end{align*}
It remains to simplify the numerator.  Since
$\widetilde b_{j+1}^{(s,m)}=a_{j+1}^{(s,m)}+b_{j+1}^{(s,m)}$ by
Lemma~\ref{lem-Gm-triangular},
\begin{align*}
&b_0^{(s,m)}\Delta_{s,m}(z)
+z\bigl((b_0^{(s,m)}+1)p_0(z)-q_0(z)\bigr) \\
&\quad = b_0^{(s,m)}\Delta_{s,m}(z)
+z\left(b_0^{(s,m)}p_0(z)-\sum_{r=0}^{s-1}(-1)^r b_{r+1}^{(s,m)}z^r\right).
\end{align*}
The first two terms combine to
\[
b_0^{(s,m)}\bigl(\Delta_{s,m}(z)+zp_0(z)\bigr)=b_0^{(s,m)},
\]
because $\Delta_{s,m}(z)=1-zp_0(z)$.  Hence the numerator is
\[
b_0^{(s,m)}-
\sum_{r=0}^{s-1}(-1)^r b_{r+1}^{(s,m)}z^{r+1}
=
\sum_{j=0}^{s}(-1)^j b_j^{(s,m)}z^j.
\]
This proves \eqref{eq-Om-general}.

\textit{Assembly.}  For $d\ge2$, the coefficients of
$\mathcal O_m^{(s)}$ and $\mathcal E_m^{(s)}$ give the odd and even
fixed-dilation counts.  The only mismatch is in dimension $1$: the constant
term of $\mathcal O_m^{(s)}$ is the auxiliary value $(s+1)m+1$, whereas the
Ehrhart convention is $L_1^{(s)}(m)=sm+1$.  Thus subtracting $m$ from the odd
series corrects precisely the coefficient of $y$.  Including the
zero-dimensional term gives
\[
G_m^{(s)}(y)=1+y(\mathcal O_m^{(s)}(y^2)-m)+y^2\mathcal E_m^{(s)}(y^2),
\]
which is \eqref{eq-Gms-general}.
\end{proof}

\begin{remark}[Specialization to the original lattice-point counts]
\label{rem-m1-specialization}
For $m=1$, the fixed-dilation series recovers the original dimension
generating function, with the low-dimensional normalization
\[
G_1^{(s)}(x)=1+(s+1)x+F_s(x).
\]
Indeed, $L_d^{(s)}(1)=N_d(s)$ for $d\ge2$.  In this specialization
\[
a_j^{(s,1)}=\binom{s+2}{2j},\qquad
b_j^{(s,1)}=\binom{s+2}{2j+1},
\]
so $\Delta_{s,1}(z)=Q_{s+2}(-z)$ and Theorem~\ref{thm-Gm-general}
recovers the rational formula for $F_s$.  Nevertheless, this specialization
only gives the rational dimension generating function.  It does not reveal the
M\"{o}bius recurrence \eqref{eq-Phi-rec-main} or the trigonometric closed
forms \eqref{eq-Phi-tan-main}--\eqref{eq-Psi-trig-main}; those use the
first-order Pascal system for $P_n$ and $Q_n$ and justify the direct analysis in
Section~\ref{sec-2}.
\end{remark}

\subsection{Consequences of the general formula}
\label{subsec-consequences}

With Theorems~\ref{thm-hstar-extraction} and \ref{thm-Gm-general} in hand,
we now derive their consequences in the same order as
Subsection~\ref{subsec-main-proofs}---recurrences, then the trigonometric/analytic
structure, then the full $h^*$-picture.

\subsubsection*{Parity recurrences and volume generating functions}

\begin{corollary}[Parity recurrences for fixed $m$]\label{cor-parity-rec}
Fix $s,m\ge1$.  Write
\[
E_r^{(s,m)}:=L_{2r+2}^{(s)}(m)\qquad (r\ge0),
\]
and
\[
O_r^{(s,m)}:=L_{2r+1}^{(s)}(m)\qquad (r\ge1).
\]
Then the even-dimensional sequence satisfies, for $r\ge s$,
\[
E_r^{(s,m)}=
\sum_{j=1}^{s}(-1)^{j+1}a_j^{(s,m)}E_{r-j}^{(s,m)}.
\]
The odd-dimensional sequence satisfies, for $r\ge s+1$,
\[
O_r^{(s,m)}=
\sum_{j=1}^{s}(-1)^{j+1}a_j^{(s,m)}O_{r-j}^{(s,m)}.
\]
Equivalently, if one uses the auxiliary initial value
\[
\widetilde O_0^{(s,m)}:=(s+1)m+1
\]
and sets
\[
\widetilde O_r^{(s,m)}:=O_r^{(s,m)}\qquad (r\ge1),
\]
then the same recurrence holds for $\widetilde O_r^{(s,m)}$ for all
$r\ge s+1$.
\end{corollary}

\begin{proof}
Write
\[
\Delta_{s,m}(z)=\sum_{j=0}^{s}(-1)^j a_j^{(s,m)}z^j,
\qquad a_0^{(s,m)}=1.
\]

For the even part, Theorem~\ref{thm-Gm-general} gives
\[
\mathcal E_m^{(s)}(z)
=
\frac{A_E(z)}{\Delta_{s,m}(z)},
\qquad
A_E(z):=\sum_{j=0}^{s-1}(-1)^j a_{j+1}^{(s,m)}z^j.
\]
Thus
\[
\Delta_{s,m}(z)\mathcal E_m^{(s)}(z)=A_E(z).
\]
Since
\[
\mathcal E_m^{(s)}(z)=\sum_{r\ge0}E_r^{(s,m)}z^r,
\]
comparison of the coefficient of $z^r$ gives
\[
\sum_{j=0}^{s}(-1)^j a_j^{(s,m)}E_{r-j}^{(s,m)}=[z^r]A_E(z),
\]
where terms with negative subscripts are omitted.  Since $\deg A_E\le s-1$,
the right-hand side is zero for $r\ge s$.  Hence
\[
E_r^{(s,m)}=
\sum_{j=1}^{s}(-1)^{j+1}a_j^{(s,m)}E_{r-j}^{(s,m)},
\qquad r\ge s.
\]

For the odd part, Theorem~\ref{thm-Gm-general} gives
\[
\mathcal O_m^{(s)}(z)
=
\frac{A_O(z)}{\Delta_{s,m}(z)},
\qquad
A_O(z):=\sum_{j=0}^{s}(-1)^j b_j^{(s,m)}z^j.
\]
The constant term of $\mathcal O_m^{(s)}(z)$ is the auxiliary value
\[
\widetilde O_0^{(s,m)}=(s+1)m+1,
\]
whereas the actual one-dimensional Ehrhart value is
$L_1^{(s)}(m)=sm+1$.  For $r\ge1$, the coefficient of $z^r$ in
$\mathcal O_m^{(s)}(z)$ is $L_{2r+1}^{(s)}(m)$.  Thus
\[
\mathcal O_m^{(s)}(z)=\sum_{r\ge0}\widetilde O_r^{(s,m)}z^r.
\]
Comparing coefficients in
\[
\Delta_{s,m}(z)\mathcal O_m^{(s)}(z)=A_O(z)
\]
gives
\[
\sum_{j=0}^{s}(-1)^j a_j^{(s,m)}\widetilde O_{r-j}^{(s,m)}
=[z^r]A_O(z),
\]
again omitting terms with negative subscripts.  Since $\deg A_O\le s$,
the right-hand side is zero for $r\ge s+1$.  Therefore
\[
\widetilde O_r^{(s,m)}=
\sum_{j=1}^{s}(-1)^{j+1}a_j^{(s,m)}\widetilde O_{r-j}^{(s,m)},
\qquad r\ge s+1.
\]
For $r\ge s+1$ and $1\le j\le s$, one has $r-j\ge1$, so
$\widetilde O_{r-j}^{(s,m)}=O_{r-j}^{(s,m)}$.  Hence the same recurrence holds
for the actual odd-dimensional sequence $O_r^{(s,m)}$ for $r\ge s+1$.
\end{proof}

\begin{corollary}[Volume generating functions]\label{cor-vol-general}
For fixed $s\ge1$, set $A_j^{(s)}:=(s+1-j)^{2j}/(2j)!$ and
$B_j^{(s)}:=(s+1-j)^{2j+1}/(2j+1)!$ for $0\le j\le s$, and let
$\Delta_s^{\mathrm{vol}}(z):=\sum_{j=0}^{s}(-1)^jA_j^{(s)}z^j$.
Define $\mathscr E_s(z)$ and $\mathscr O_s(z)$ as the limits of
$m^{-2}\mathcal E_m^{(s)}(z/m^2)$ and $m^{-1}\mathcal O_m^{(s)}(z/m^2)$
as $m\to\infty$, where each limit is taken coefficientwise in $z$:
for each fixed power of $z$, extract the coefficient as a rational function
of $m$ and take its leading term, then reassemble into a formal power series.
Then
\[
\mathscr E_s(z) = \frac{\sum_{j=0}^{s-1}(-1)^jA_{j+1}^{(s)}z^j}{\Delta_s^{\mathrm{vol}}(z)},
\qquad
\mathscr O_s(z) = \frac{\sum_{j=0}^{s}(-1)^jB_j^{(s)}z^j}{\Delta_s^{\mathrm{vol}}(z)},
\]
and the volume generating function
$\mathcal V_s(y):=\sum_{d\ge0}\vol(\cP_d^{(s)})\,y^d$ satisfies
$\mathcal V_s(y)=1+y(\mathscr O_s(y^2)-1)+y^2\mathscr E_s(y^2)$.
\end{corollary}

\begin{proof}
We take all limits coefficientwise in $z$.  First observe that, for each fixed
$j$,
\[
a_j^{(s,m)}
=\binom{(s+1-j)m+j+1}{2j}
=A_j^{(s)}m^{2j}+O(m^{2j-1}),
\]
and
\[
b_j^{(s,m)}
=\binom{(s+1-j)m+j+1}{2j+1}
=B_j^{(s)}m^{2j+1}+O(m^{2j}).
\]
Therefore, after substituting $z\mapsto z/m^2$, we have
\[
\Delta_{s,m}(z/m^2)
=\sum_{j=0}^{s}(-1)^j a_j^{(s,m)}\frac{z^j}{m^{2j}}
\longrightarrow
\sum_{j=0}^{s}(-1)^jA_j^{(s)}z^j
=\Delta_s^{\mathrm{vol}}(z).
\]
The denominator has constant term $1$ for every $m$, so taking reciprocals is
legitimate as a coefficientwise operation on formal power series.

For the even part, Theorem~\ref{thm-Gm-general} gives
\[
\mathcal E_m^{(s)}(z)
=
\frac{\sum_{j=0}^{s-1}(-1)^ja_{j+1}^{(s,m)}z^j}{\Delta_{s,m}(z)}.
\]
Thus
\begin{align*}
\frac{1}{m^2}\mathcal E_m^{(s)}(z/m^2)
&=
\frac{\sum_{j=0}^{s-1}(-1)^j
        a_{j+1}^{(s,m)}z^j/m^{2j+2}}
     {\Delta_{s,m}(z/m^2)} \\
&\longrightarrow
\frac{\sum_{j=0}^{s-1}(-1)^jA_{j+1}^{(s)}z^j}
     {\Delta_s^{\mathrm{vol}}(z)}.
\end{align*}
This is the asserted formula for $\mathscr E_s(z)$.

For the odd part, similarly,
\[
\mathcal O_m^{(s)}(z)
=
\frac{\sum_{j=0}^{s}(-1)^jb_j^{(s,m)}z^j}{\Delta_{s,m}(z)},
\]
and hence
\begin{align*}
\frac{1}{m}\mathcal O_m^{(s)}(z/m^2)
&=
\frac{\sum_{j=0}^{s}(-1)^j
        b_j^{(s,m)}z^j/m^{2j+1}}
     {\Delta_{s,m}(z/m^2)} \\
&\longrightarrow
\frac{\sum_{j=0}^{s}(-1)^jB_j^{(s)}z^j}
     {\Delta_s^{\mathrm{vol}}(z)}.
\end{align*}
This proves the formula for $\mathscr O_s(z)$.

It remains to identify the assembled series with the volume generating
function.  For $r\ge0$,
\[
[z^r]\mathscr E_s(z)=\vol(\cP_{2r+2}^{(s)}),
\]
because $[z^r]\mathcal E_m^{(s)}(z)=L_{2r+2}^{(s)}(m)$ and
$L_{2r+2}^{(s)}(m)/m^{2r+2}$ tends to the Euclidean volume of
$\cP_{2r+2}^{(s)}$.  Likewise, for $r\ge1$,
\[
[z^r]\mathscr O_s(z)=\vol(\cP_{2r+1}^{(s)}).
\]
For $r=0$, however, the constant term of $\mathcal O_m^{(s)}$ is the auxiliary
quantity $(s+1)m+1$, so $[z^0]\mathscr O_s(z)=s+1$.  The actual
one-dimensional volume is $\vol(\cP_1^{(s)})=s$.  This explains the correction
term in
\[
\mathcal V_s(y)=1+y(\mathscr O_s(y^2)-1)+y^2\mathscr E_s(y^2).
\]
The constant term $1$ is the volume of the zero-dimensional convention.  Hence
the displayed expression gives exactly
$\sum_{d\ge0}\vol(\cP_d^{(s)})y^d$.
\end{proof}

\subsubsection*{Bivariate master identity and low-degree $h^*$-formulas}

\begin{proposition}[Bivariate master identity]
\label{prop-H-master}
Define the two-variable generating function
$\mathcal H_s(y,z):=\sum_{k\ge0}H_k^{(s)}(y)z^k=\sum_{d\ge0}h_d^{(s)}(z)\,y^d$,
where $h_d^{(s)}(z):=\sum_{k=0}^{d}h_{d,k}^{(s)}z^k$.
Then
\begin{equation}\label{eq-H-master}
\mathcal H_s(y,z)
=(1-z)\sum_{m\ge0}G_m^{(s)}\bigl((1-z)y\bigr)z^m.
\end{equation}
Equivalently,
$H_k^{(s)}(y)=\bigl[z^k\bigr]\bigl((1-z)\sum_{m\ge0}G_m^{(s)}((1-z)y)z^m\bigr)$.
\end{proposition}

\begin{proof}
We regard all series as formal power series in $z$ whose coefficients are
formal power series in $y$.  Starting from the definition of $G_m^{(s)}$, we
have
\begin{align*}
\sum_{m\ge0}G_m^{(s)}\bigl((1-z)y\bigr)z^m
&=\sum_{m\ge0}\sum_{d\ge0}L_d^{(s)}(m)\bigl((1-z)y\bigr)^d z^m \\
&=\sum_{d\ge0}y^d(1-z)^d\sum_{m\ge0}L_d^{(s)}(m)z^m.
\end{align*}
For each fixed $d$, the inner sum is the Ehrhart series
\[
\sum_{m\ge0}L_d^{(s)}(m)z^m
=\frac{h_d^{(s)}(z)}{(1-z)^{d+1}},
\]
where for $d=0,1$ this follows from the low-dimensional conventions stated in
Subsection~\ref{subsec-ehrhart-setup}, and for $d\ge2$ it is exactly
\eqref{eq-Ehr-ds}.  Hence
\begin{align*}
\sum_{m\ge0}G_m^{(s)}\bigl((1-z)y\bigr)z^m
&=\sum_{d\ge0}y^d(1-z)^d\frac{h_d^{(s)}(z)}{(1-z)^{d+1}}\\
&=\frac{1}{1-z}\sum_{d\ge0}h_d^{(s)}(z)y^d
=\frac{\mathcal H_s(y,z)}{1-z}.
\end{align*}
Multiplying by $1-z$ gives \eqref{eq-H-master}.  Extracting the coefficient of
$z^k$ on both sides gives the final equivalent formulation for $H_k^{(s)}(y)$.
\end{proof}

\begin{proposition}[Explicit formulas for $H_k^{(s)}(y)$ in low degree]
\label{prop-Hk-low}
For every $s\ge1$,
\begin{align*}
H_0^{(s)}(y)
&= \frac{1}{1-y},\\[4pt]
H_1^{(s)}(y)
&= G_1^{(s)}(y)-\frac{1}{(1-y)^2},\\[4pt]
H_2^{(s)}(y)
&= G_2^{(s)}(y)-\frac{d}{dy}\!\Bigl(yG_1^{(s)}(y)\Bigr)+\frac{y}{(1-y)^3},\\[4pt]
H_3^{(s)}(y)
&= G_3^{(s)}(y)-\frac{d}{dy}\!\Bigl(yG_2^{(s)}(y)\Bigr)
   +\frac{y}{2}\frac{d^2}{dy^2}\!\Bigl(yG_1^{(s)}(y)\Bigr)-\frac{y^2}{(1-y)^4},\\[4pt]
H_4^{(s)}(y)
&= G_4^{(s)}(y)-\frac{d}{dy}\!\Bigl(yG_3^{(s)}(y)\Bigr)
   +\frac{y}{2}\frac{d^2}{dy^2}\!\Bigl(yG_2^{(s)}(y)\Bigr)
   -\frac{y^2}{6}\frac{d^3}{dy^3}\!\Bigl(yG_1^{(s)}(y)\Bigr)+\frac{y^3}{(1-y)^5}.
\end{align*}
\end{proposition}

\begin{proof}
We apply Theorem~\ref{thm-hstar-extraction} with $k=0,1,2,3,4$.  Recall that
$G_0^{(s)}(y)=\sum_{d\ge0}L_d^{(s)}(0)y^d=1/(1-y)$, because every dilated
polytope $0\cP_d^{(s)}$ consists of the origin.  The $i=k$ term in
\eqref{eq-Hk-extraction} therefore contributes
\[
(-1)^k\frac{1}{k!}y^{k-1}\frac{d^k}{dy^k}\left(\frac{y}{1-y}\right).
\]
For $k\ge1$,
\[
\frac{1}{k!}y^{k-1}\frac{d^k}{dy^k}\left(\frac{y}{1-y}\right)
=\frac{y^{k-1}}{(1-y)^{k+1}}.
\]
For $k=0$, the same formula is interpreted simply as $G_0^{(s)}(y)=1/(1-y)$.

Now substitute successively.
For $k=1$,
\[
H_1^{(s)}(y)=G_1^{(s)}(y)-\frac{d}{dy}\left(\frac{y}{1-y}\right)
=G_1^{(s)}(y)-\frac{1}{(1-y)^2}.
\]
For $k=2$,
\[
H_2^{(s)}(y)=G_2^{(s)}(y)-\frac{d}{dy}\bigl(yG_1^{(s)}(y)\bigr)
+\frac{y}{2}\frac{d^2}{dy^2}\left(\frac{y}{1-y}\right),
\]
and the last term equals $y/(1-y)^3$.  The cases $k=3$ and $k=4$ are obtained
in exactly the same way, with the signs alternating as in
\eqref{eq-Hk-extraction}.  This gives the five displayed formulas.
\end{proof}

\begin{corollary}[Coefficient formulas in low degree]
\label{cor-hdk-low}
For every $d\ge0$ and $s\ge1$,
\begin{align*}
h_{d,1}^{(s)} &=L_d^{(s)}(1)-(d+1),\\
h_{d,2}^{(s)} &=L_d^{(s)}(2)-(d+1)L_d^{(s)}(1)+\tbinom{d+1}{2},\\
h_{d,3}^{(s)} &=L_d^{(s)}(3)-(d+1)L_d^{(s)}(2)+\tbinom{d+1}{2}L_d^{(s)}(1)-\tbinom{d+1}{3},\\
h_{d,4}^{(s)} &=L_d^{(s)}(4)-(d+1)L_d^{(s)}(3)+\tbinom{d+1}{2}L_d^{(s)}(2)
  -\tbinom{d+1}{3}L_d^{(s)}(1)+\tbinom{d+1}{4}.
\end{align*}
More generally, $h_{d,k}^{(s)}=\sum_{m=0}^{k}(-1)^{k-m}\binom{d+1}{k-m}L_d^{(s)}(m)$.
\end{corollary}

\begin{proof}
For $d\ge2$, these are the coefficient forms of
Lemma~\ref{lem-hstar-inversion}, after the change of variables $m=k-i$.
For $d=0,1$, the formulas follow directly from the conventions
$L_0^{(s)}(m)=1$ and $L_1^{(s)}(m)=sm+1$, as in the proof of
Theorem~\ref{thm-hstar-extraction}.
\end{proof}

\begin{corollary}[Rationality and denominator bound for $H_k^{(s)}(y)$]
\label{cor-Hk-denominator}
For every fixed $k\ge0$ and $s\ge1$, $H_k^{(s)}(y)$ is a rational
function of $y$, and with $D_{s,m}(y):=\Delta_{s,m}(y^2)$ for $m\ge1$,
one may take the denominator of $H_k^{(s)}(y)$ to divide
\begin{equation}\label{eq-Hk-denominator-bound}
(1-y)^{k+1}\prod_{m=1}^{k}D_{s,m}(y)^{k-m+1}.
\end{equation}
\end{corollary}

\begin{proof}
Rationality follows from Theorem~\ref{thm-hstar-extraction} together with
Theorem~\ref{thm-Gm-general} and the identity $G_0^{(s)}(y)=1/(1-y)$.
For the denominator bound: the term involving
$G_m^{(s)}$ in \eqref{eq-Hk-extraction} is differentiated exactly $k-m$
times. If $F(y)/D(y)^r$ is differentiated once, the result has denominator
$D(y)^{r+1}$, so each differentiation raises the denominator exponent by at
most $1$. Since $G_m^{(s)}$ has denominator $D_{s,m}(y)$ for $m\ge1$ and
$G_0^{(s)}(y)=1/(1-y)$, the bound \eqref{eq-Hk-denominator-bound} follows.
\end{proof}

\begin{remark}[On the possibility of a closed form]
\label{rem-Hk-closed-form}
Proposition~\ref{prop-H-master} is the cleanest general formula known for
$\{H_k^{(s)}(y)\}_{k\ge0}$. Unlike $\Phi_s$ and $\Psi_s$, a fixed
$H_k^{(s)}(y)$ mixes the varying denominators $\Delta_{s,m}(y^2)$ for
$1\le m\le k$, so no single-denominator formula is apparent. Any future
closed form for the bivariate series $\mathcal H_s(y,z)$ would immediately
yield simultaneous formulas for all $H_k^{(s)}(y)$.
\end{remark}

\subsection{Specialization to \texorpdfstring{$s=1$}{s=1}: explicit formulas and Gorenstein proof}
\label{subsec-s1}

We now specialize to $s=1$, playing the role that the small-cases example
(\S\ref{subsec-small-cases}) played in Section~\ref{sec-2}: it verifies all
general formulas concretely, gives the simplest closed forms, and yields the
Gorenstein proof.

\subsubsection*{Setup}

Set $\mathcal T_m:=\{(a,b)\in\R_{\ge0}^2:a+b\le m\}$ (so $\mathcal T_m=m\mathcal T_1$),
$\Gamma_m:=\{(a,b,c)\in\R_{\ge0}^3:a+b\le m,\,b+c\le2m\}$, and
\[
\Lambda_m:=\#(\mathcal T_m\cap\Z^2)=\binom{m+2}{2},
\qquad
\Theta_m:=\#(\Gamma_m\cap\Z^3)=\frac{(m+1)(m+2)(5m+3)}{6}.
\]
The formula for $\Theta_m$ follows from $\Theta_m=\sum_{b=0}^m(m-b+1)(2m-b+1)$.

\subsubsection*{The closed-form formula for $G_m^{(1)}(y)$}

\begin{proposition}[Closed form for $G_m^{(1)}$]\label{prop-Gm-s1}
For every integer $m\ge0$,
\begin{equation}\label{eq-Gm-s1-closed}
G_m^{(1)}(y)
=\frac{1+(m+1)y+\dfrac{m(m+1)(m+2)}{3}\,y^3}{1-\binom{m+2}{2}y^2}.
\end{equation}
\end{proposition}

We give two proofs: the first is geometric and provides the product
decomposition needed for the Gorenstein theorem; the second verifies
consistency with Theorem~\ref{thm-Gm-general}.

\begin{proof}[First proof: geometric decomposition]
\noindent\textit{Even dimensions $d=2r$, $r\ge1$.}
In $m\cP_{2r}^{(1)}$, each even-indexed constraint $x_{2j}+x_{2j+1}\le2m$
($j=1,\ldots,r-1$) is redundant: the odd constraints to its left and right
force $x_{2j}\le m$ and $x_{2j+1}\le m$, whence the even constraint holds.
The last constraint is odd ($x_{2r-1}+x_{2r}\le m$), so the active system
is exactly $\mathcal T_m^r$, giving $L_{2r}^{(1)}(m)=\Lambda_m^r$.

\noindent\textit{Odd dimensions $d=2r+1$, $r\ge1$.}
The same redundancy applies to all even-indexed constraints except the last
($x_{2r}+x_{2r+1}\le2m$, which has no right neighbor). The active system is
$\mathcal T_m^{r-1}\times\Gamma_m$, giving $L_{2r+1}^{(1)}(m)=\Theta_m\Lambda_m^{r-1}$.

Combining with $L_0^{(1)}(m)=1$ and $L_1^{(1)}(m)=m+1$:
\[
G_m^{(1)}(y)=1+(m+1)y+\frac{\Lambda_m y^2+\Theta_m y^3}{1-\Lambda_m y^2}.
\]
The identity $\Theta_m-(m+1)\Lambda_m=m(m+1)(m+2)/3$ (verified by direct
expansion) yields \eqref{eq-Gm-s1-closed}.
\end{proof}

\begin{proof}[Second proof: specialization of Theorem~\ref{thm-Gm-general}]
For $s=1$, the coefficients of Theorem~\ref{thm-Gm-general} reduce to
\[
a_0^{(1,m)}=1,\quad a_1^{(1,m)}=\tbinom{m+2}{2},\quad
b_0^{(1,m)}=2m+1,\quad b_1^{(1,m)}=\tbinom{m+2}{3},
\]
so $\Delta_{1,m}(z)=1-\binom{m+2}{2}z$ and
\[
\mathcal E_m^{(1)}(z)=\frac{\tbinom{m+2}{2}}{1-\tbinom{m+2}{2}z},
\qquad
\mathcal O_m^{(1)}(z)=\frac{(2m+1)-\tbinom{m+2}{3}z}{1-\tbinom{m+2}{2}z}.
\]
Substituting into $G_m^{(1)}(y)=1+y(\mathcal O_m^{(1)}(y^2)-m)+y^2\mathcal E_m^{(1)}(y^2)$
and simplifying yields \eqref{eq-Gm-s1-closed}.
\end{proof}

\subsubsection*{Geometric consequences}

\begin{corollary}[Lattice-point counts and volumes for $s=1$]\label{cor-volume}
For every integer $r\ge1$:
\begin{enumerate}[label=\upshape(\alph*)]
\item $N_{2r}(1)=3^r$ and $N_{2r+1}(1)=8\cdot3^{r-1}$.
\item $(2r)!\,\vol(\cP_{2r}^{(1)})=(2r)!/2^r$ and
$(2r+1)!\,\vol(\cP_{2r+1}^{(1)})=5(2r+1)!/(6\cdot2^{r-1})$.
\end{enumerate}
\end{corollary}

\begin{proof}
For (a), specialize the product formulas in the first proof of
Proposition~\ref{prop-Gm-s1} to $m=1$.  Since
\[
\Lambda_1=\#(\mathcal T_1\cap\Z^2)=\binom{3}{2}=3,
\qquad
\Theta_1=\#(\Gamma_1\cap\Z^3)=\frac{2\cdot3\cdot8}{6}=8,
\]
we get
\[
N_{2r}(1)=L_{2r}^{(1)}(1)=\Lambda_1^r=3^r,
\]
and, for $r\ge1$,
\[
N_{2r+1}(1)=L_{2r+1}^{(1)}(1)=\Theta_1\Lambda_1^{r-1}=8\cdot3^{r-1}.
\]

For (b), use the geometric decompositions
\[
\cP_{2r}^{(1)}=\mathcal T_1^r,
\qquad
\cP_{2r+1}^{(1)}=\mathcal T_1^{r-1}\times\Gamma_1
\]
from the first proof of Proposition~\ref{prop-Gm-s1}.  Volumes multiply under
Cartesian products.  Since $\vol(\mathcal T_1)=1/2$ and
\[
\vol(\Gamma_1)=\int_0^1(1-b)(2-b)\,db=\frac56,
\]
we obtain
\[
\vol(\cP_{2r}^{(1)})=2^{-r},
\qquad
\vol(\cP_{2r+1}^{(1)})=\frac56\,2^{-(r-1)}.
\]
Multiplying by $(2r)!$ and $(2r+1)!$, respectively, gives the normalized
volumes in the statement.
\end{proof}

\begin{remark}\label{rem-product}
The product rule $L_{2r}^{(1)}(m)=L_{\mathcal T_1}(m)^r$ follows because the
Ehrhart polynomial of a Cartesian product is the product of the factors'
Ehrhart polynomials. The normalized volumes
$(2r)!/2^r=1,6,90,2520,\ldots$ equal $(2r-1)!!\cdot r!$.
\end{remark}

\subsubsection*{The \texorpdfstring{$h^*$}{h*}-vectors and the Gorenstein proof}

Applying Proposition~\ref{prop-Hk-low} to the closed form \eqref{eq-Gm-s1-closed}
yields explicit generating functions for the $h^*$-coefficients.

\begin{corollary}[$h^*$-coefficient generating functions for $s=1$]\label{cor-Hk-s1}
For $s=1$, the first four series are
\begin{align*}
H_0^{(1)}(y) &= \frac{1}{1-y},\\[4pt]
H_1^{(1)}(y) &= \frac{2y^3(y^2-2y+2)}{(1-y)^2(1-3y^2)},\\[4pt]
H_2^{(1)}(y) &= \frac{y^4(1+24y-42y^2+36y^3-39y^4)}{(1-y)^3(1-3y^2)^2(1-6y^2)},\\[4pt]
H_3^{(1)}(y) &= \frac{4y^5\mathcal{N}_3(y)}{(1-y)^4(1-3y^2)^3(1-6y^2)^2(1-10y^2)},
\end{align*}
where $\mathcal{N}_3(y)$ is the degree-$10$ polynomial
\begin{align*}
\mathcal{N}_3(y) &=
1+y+73y^2-185y^3+38y^4-84y^5\\
&\quad-1077y^6+4014y^7-3717y^8+1944y^9-1458y^{10}.
\end{align*}
\end{corollary}

\begin{proof}
Substitute the closed form for $G_m^{(1)}(y)$ from
Proposition~\ref{prop-Gm-s1} into the general formulas of
Proposition~\ref{prop-Hk-low} for $k=0,1,2,3$, and then simplify the resulting
rational functions.  No additional geometric input is used here; the corollary
is only the specialization of Proposition~\ref{prop-Hk-low} to $s=1$.
\end{proof}

\begin{remark}
The same substitution also gives a formula for $H_4^{(1)}(y)$, but its
numerator is lengthy.  Since the Gorenstein proof below uses the product
decomposition of $\cP_{2r}^{(1)}$ rather than this long rational function, we
omit the full expression.
\end{remark}

For reference, the same Ehrhart inversion computation gives the small-dimensional data in Table~\ref{tab-hstar-s1}.

\begin{table}[htbp]
\centering
\caption{The $h^*$-vectors of $\cP_d^{(1)}$ for $2\le d\le8$, computed
from the Ehrhart inversion formula \eqref{eq-hstar-inversion}.  The displayed
rows are consistency checks for the formulas above; the Gorenstein result is
proved independently from the product decomposition.  All unspecified
coefficients are zero.}
\label{tab-hstar-s1}
\smallskip
\begin{tabular}{c|cccccccc|c}
\toprule
$d$ & $h_{d,0}^{(1)}$ & $h_{d,1}^{(1)}$ & $h_{d,2}^{(1)}$ & $h_{d,3}^{(1)}$ & $h_{d,4}^{(1)}$ & $h_{d,5}^{(1)}$ & $h_{d,6}^{(1)}$ & $h_{d,7}^{(1)}$ & $\sum_k h_{d,k}^{(1)}$ \\
\midrule
2 & 1 & 0 & 0 & 0 & 0 & 0 & 0 & 0 & 1   \\
3 & 1 & 4 & 0 & 0 & 0 & 0 & 0 & 0 & 5   \\
4 & 1 & 4 & 1 & 0 & 0 & 0 & 0 & 0 & 6   \\
5 & 1 & 18 & 27 & 4 & 0 & 0 & 0 & 0 & 50  \\
6 & 1 & 20 & 48 & 20 & 1 & 0 & 0 & 0 & 90  \\
7 & 1 & 64 & 388 & 472 & 121 & 4 & 0 & 0 & 1050\\
8 & 1 & 72 & 603 & 1168 & 603 & 72 & 1 & 0 & 2520\\
\bottomrule
\end{tabular}
\end{table}

Several structural features are worth recording.

\begin{enumerate}[label=(\alph*)]
\item \textit{Nonnegativity}: All $h^*$-coefficients are nonneg\-ative, as
guaranteed by Stanley's theorem \cite{Stanley1978}. The data suggests
$h_{d,k}^{(1)}>0$ whenever $0\le k\le d-1$ and $d\ge3$.

\item \textit{Palindromicity in even dimensions}: The nonzero $h^*$-vectors for
$d=4,6,8$ are palindromic:
$(1,4,1)$, $(1,20,48,20,1)$, $(1,72,603,1168,603,72,1)$.
These are fully explained by Theorem~\ref{thm-gorenstein-s1}.

\item \textit{Non-palindromicity in odd dimensions}: Consistent with the
asymmetry $\cP_{2r+1}^{(1)}=\mathcal T_1^{r-1}\times\Gamma_1$ where
$\Gamma_1\not\cong \mathcal T_1$.

\item \textit{OEIS match}: For $d=4,6,8$ the $h^*$-vectors appear as initial
rows of OEIS A154283 \cite{OEIS}.
\end{enumerate}

\begin{proof}[Proof of Theorem~\ref{thm-gorenstein-s1}]
The first proof of Proposition~\ref{prop-Gm-s1} gives the set identity
$\cP_{2r}^{(1)}=\mathcal T_1^r$: when $s=1$, all even-indexed constraints are
redundant in even ambient dimension, and the only active constraints are the
$r$ inequalities $x_{2j-1}+x_{2j}\le1$ ($1\le j\le r$), which describe
$\mathcal T_1^r$ exactly. It therefore suffices to prove $\mathcal T_1^r$ is
Gorenstein.  We use the standard interior-lattice-point characterization of
Gorenstein lattice polytopes; see, for example, \cite{hibi-1,Bruns}.

A lattice point $(u,v)$ lies in $\operatorname{int}((n+3)\mathcal T_1)$ iff
$u,v\ge1$ and $u+v\le n+2$, i.e., $(u-1,v-1)\in n\mathcal T_1\cap\Z^2$.
Hence $\operatorname{int}((n+3)\mathcal T_1)\cap\Z^2=(1,1)+(n\mathcal T_1\cap\Z^2)$
for all $n\ge0$, so $\mathcal T_1$ is Gorenstein of index~$3$.

Setting $c_r:=((1,1),\ldots,(1,1))\in\Z^{2r}$, the factorization of interior
and lattice points across Cartesian products gives
$\operatorname{int}((n+3)\mathcal T_1^r)\cap\Z^{2r}=c_r+(n\mathcal T_1^r\cap\Z^{2r})$
for all $n\ge0$. Hence $\mathcal T_1^r$ is Gorenstein of index~$3$, and so is
$\cP_{2r}^{(1)}=\mathcal T_1^r$.
\end{proof}

\begin{remark}[Non-Gorenstein behavior for $s\ge2$]\label{rem-not-gorenstein}
Theorem~\ref{thm-gorenstein-s1} is genuinely exceptional: by
Corollary~\ref{cor-not-gorenstein-all-s}, for every $s\ge2$ at least one
even dimension fails to be Gorenstein. The computed $h^*$-vectors include:
\begin{center}
\begin{tabular}{c|cccc}
\toprule
 & $d=2$ & $d=4$ & $d=6$ & $d=8$ \\
\midrule
$s=2$ & $[1,3]$ & $[1,30,55,9]$ & $[1,197,1818,2786,811,27]$ & $\cdots$ \\
$s=3$ & $[1,7,1]^{\dag}$ & $[1,90,284,94,1]$ & $[1,893,13714,31436,14273,973,1]$ & $\cdots$ \\
$s=4$ & $[1,12,3]$ & $[1,205,859,381,9]$ & $\cdots$ & $\cdots$ \\
\bottomrule
\end{tabular}
\end{center}
The marked case is Gorenstein because $\cP_2^{(3)}=\mathcal T_3$
(Proposition~\ref{prop-d2-gorenstein}); all others are not.
Our computations show $\cP_{2r}^{(2)}$ is not Gorenstein for $r\le6$, and
for $s=3$ only $r=1$ is Gorenstein among $r\le4$.
\end{remark}

\begin{remark}[Unimodality and real-rootedness]\label{rem-unimodal}
The $h^*$-vectors in Table~\ref{tab-hstar-s1} are unimodal for all displayed
$d$. Whether $h_d^{(s)}(z)$ is real-rooted, or at least unimodal, for all
$d$ and $s$ remains open; see \cite{BrandenSolus,Braun}.
\end{remark}

\begin{remark}[Small even-dimensional $h^*$-polynomials for $s=1$]\label{rem-small-gamma}
The first even-dimensional cases in Table~\ref{tab-hstar-s1} are palindromic
and unimodal, in agreement with Theorem~\ref{thm-gorenstein-s1}.  Expressed in
the standard symmetric basis, they have the following nonnegative expansions:
\begin{center}
\begin{tabular}{ccl}
\toprule
$r$ & $h^*(\cP_{2r}^{(1)},z)$ & symmetric-basis coefficients\\
\midrule
$1$ & $1$ & $(1)$\\
$2$ & $1+4z+z^2$ & $(1,2)$\\
$3$ & $1+20z+48z^2+20z^3+z^4$ & $(1,16,10)$\\
$4$ & $1+72z+603z^2+1168z^3+603z^4+72z^5+z^6$ & $(1,66,324,104)$\\
\bottomrule
\end{tabular}
\end{center}
We do not use these computations in the proof of the Gorenstein theorem.  A
uniform proof of real-rootedness or $\gamma$-positivity for all $r$ would be an
interesting strengthening; see the unimodality question in Section~\ref{sec-conclusion}.
\end{remark}

\subsubsection*{The first few Ehrhart series}

For reference:
\begin{align*}
\Ehr_{\cP_2^{(1)}}(z) &= \frac{1}{(1-z)^3},
\quad\bigl[h^*(1)=1=2!/2\bigr],\\[2pt]
\Ehr_{\cP_3^{(1)}}(z) &= \frac{1+4z}{(1-z)^4},
\quad\bigl[h^*(1)=5\bigr],\\[2pt]
\Ehr_{\cP_4^{(1)}}(z) &= \frac{1+4z+z^2}{(1-z)^5},
\quad\bigl[h^*(1)=6=4!/4\bigr],\\[2pt]
\Ehr_{\cP_5^{(1)}}(z) &= \frac{1+18z+27z^2+4z^3}{(1-z)^6},
\quad\bigl[h^*(1)=50\bigr],\\[2pt]
\Ehr_{\cP_6^{(1)}}(z) &= \frac{1+20z+48z^2+20z^3+z^4}{(1-z)^7},
\quad\bigl[h^*(1)=90=6!/8\bigr],\\[2pt]
\Ehr_{\cP_8^{(1)}}(z) &= \frac{1+72z+603z^2+1168z^3+603z^4+72z^5+z^6}{(1-z)^9},
\quad\bigl[h^*(1)=2520=8!/16\bigr].
\end{align*}
Note $\cP_2^{(1)}=\mathcal T_1$. Palindromicity for $d=4,6,8$ is apparent;
normalized volumes agree with Corollary~\ref{cor-volume}(b).

\subsection{Non-Gorenstein behavior for \texorpdfstring{$s\ge2$}{s>=2}}\label{subsec-non-gorenstein}
\label{subsec-gorenstein}

We now turn to the complementary phenomenon. Section~\ref{subsec-s1}
showed that $s=1$ yields an infinite Gorenstein family, but for every $s\ge2$
the general formulas of \S\S\ref{subsec-fixedm-transfer}--\ref{subsec-consequences}
show that the family fails to be Gorenstein in at least one even dimension.
By the Stanley--Hibi theorem \cite{hibi-1,Stanley1978}, this is equivalent to
non-palindromicity of the corresponding $h^*$-polynomial.

\begin{proposition}[Gorenstein structure in dimension $2$]\label{prop-d2-gorenstein}
For every $s\ge1$, $\cP_2^{(s)}=\mathcal T_s:=\{(x_1,x_2)\in\R_{\ge0}^2:x_1+x_2\le s\}$ (the only active constraint is
$x_1+x_2\le s$), so $L_2^{(s)}(m)=\binom{sm+2}{2}$ and
\begin{equation}\label{eq-hstar-d2-general}
h^*(\cP_2^{(s)},z)
= 1+\frac{s^2+3s-4}{2}\,z+\frac{(s-1)(s-2)}{2}\,z^2.
\end{equation}
The polynomial \eqref{eq-hstar-d2-general} is palindromic if and only if
$s\in\{1,3\}$. Hence $\cP_2^{(s)}$ is Gorenstein if and only if $s\in\{1,3\}$.
\end{proposition}

\begin{proof}
In dimension $2$ the defining system contains only the inequality
$x_1+x_2\le s$ together with $x_1,x_2\ge0$, so
$\cP_2^{(s)}=\mathcal T_s$.  The Ehrhart polynomial is therefore
\[
L_2^{(s)}(m)=\#\{(u,v)\in\Z_{\ge0}^2:u+v\le sm\}=\binom{sm+2}{2}.
\]
Using the inversion formula \eqref{eq-hstar-inversion} with $d=2$, we get
\[
h_0^*=L_2^{(s)}(0)=1,
\]
\[
h_1^*=L_2^{(s)}(1)-3L_2^{(s)}(0)
=\binom{s+2}{2}-3=\frac{s^2+3s-4}{2},
\]
and
\[
h_2^*=L_2^{(s)}(2)-3L_2^{(s)}(1)+3L_2^{(s)}(0)
=\binom{2s+2}{2}-3\binom{s+2}{2}+3
=\frac{(s-1)(s-2)}{2}.
\]
This proves \eqref{eq-hstar-d2-general}.

If the $h^*$-polynomial has degree $2$, palindromicity requires
$h_0^*=h_2^*$, hence
\[
1=\frac{(s-1)(s-2)}{2},
\]
which gives $s=3$ among positive integers.  When $s=1$, the coefficient
$h_2^*$ is zero and the polynomial is simply $1$, so it is palindromic of
degree $0$.  For $s=2$, the polynomial is $1+3z$, not palindromic.  Thus
palindromicity occurs exactly for $s\in\{1,3\}$.  By the Stanley--Hibi
criterion recalled at the beginning of this subsection, this is equivalent to
the Gorenstein property for these two-dimensional lattice polytopes.
\end{proof}

\begin{proposition}[Non-Gorenstein in dimension $4$ for $s=3$]\label{prop-d4-s3-not-gorenstein}
One has $h^*(\cP_4^{(3)},z)=1+90z+284z^2+94z^3+z^4$.
In particular, $\cP_4^{(3)}$ is not Gorenstein.
\end{proposition}

\begin{proof}
For $s=3$, Theorem~\ref{thm-Gm-general} gives the even-dimensional series
\[
\mathcal E_m^{(3)}(z)
=\frac{a_1^{(3,m)}-a_2^{(3,m)}z+a_3^{(3,m)}z^2}{
1-a_1^{(3,m)}z+a_2^{(3,m)}z^2-a_3^{(3,m)}z^3}.
\]
The coefficient of $y^4$ in $G_m^{(3)}(y)$ is the coefficient of $z$ in
$\mathcal E_m^{(3)}(z)$, because $G_m^{(3)}(y)$ contains
$y^2\mathcal E_m^{(3)}(y^2)$.  Expanding the rational function to first order
therefore gives
\[
L_4^{(3)}(m)=[z]\mathcal E_m^{(3)}(z)
=(a_1^{(3,m)})^2-a_2^{(3,m)}.
\]
Here
\[
a_1^{(3,m)}=\binom{3m+2}{2},
\qquad
a_2^{(3,m)}=\binom{2m+3}{4}.
\]
Thus
\[
L_4^{(3)}(m)=\binom{3m+2}{2}^2-\binom{2m+3}{4}
=\frac{235}{12}m^4+\frac{77}{2}m^3+\frac{329}{12}m^2+\frac{17}{2}m+1.
\]
Applying \eqref{eq-hstar-inversion} with $d=4$ gives
\[
h^*(\cP_4^{(3)},z)=1+90z+284z^2+94z^3+z^4.
\]
This polynomial is not palindromic because the coefficients of $z$ and $z^3$
are different.  Hence $\cP_4^{(3)}$ is not Gorenstein by the Stanley--Hibi
palindromicity criterion.
\end{proof}

\begin{proposition}[Non-Gorenstein in dimension $6$ for $s=3$]\label{prop-d6-s3-not-gorenstein}
One has
\[
h^*(\cP_6^{(3)},z)=1+893z+13714z^2+31436z^3+14273z^4+973z^5+z^6.
\]
In particular, $\cP_6^{(3)}$ is not Gorenstein.
\end{proposition}

\begin{proof}
By Theorem~\ref{thm-Gm-general} with $s=3$, the denominator of
$\mathcal E_m^{(3)}(z)$ is
\[
\Delta_{3,m}(z)=1-a_1z+a_2z^2-a_3z^3,
\]
where
\[
a_1=a_1^{(3,m)}=\binom{3m+2}{2},
\qquad
a_2=a_2^{(3,m)}=\binom{2m+3}{4},
\qquad
a_3=a_3^{(3,m)}=\binom{m+4}{6}.
\]
The numerator of $\mathcal E_m^{(3)}(z)$ is $a_1-a_2z+a_3z^2$.  Since
\[
\frac{1}{1-a_1z+a_2z^2-a_3z^3}
=1+a_1z+(a_1^2-a_2)z^2+O(z^3),
\]
the coefficient of $z^2$ in $\mathcal E_m^{(3)}(z)$ is
\[
[z^2]\mathcal E_m^{(3)}(z)
=a_1(a_1^2-a_2)-a_1a_2+a_3=a_1^3-2a_1a_2+a_3.
\]
Expanding the expression above and simplifying gives
\begin{equation}\label{eq-L6s3}
L_6^{(3)}(m)
=\frac{61291}{720}m^6+\frac{19951}{80}m^5+\frac{42959}{144}m^4+\frac{9007}{48}m^3
+\frac{23777}{360}m^2+\frac{187}{15}m+1.
\end{equation}
Applying \eqref{eq-hstar-inversion} with $d=6$ to \eqref{eq-L6s3} yields
\[
h^*(\cP_6^{(3)},z)=1+893z+13714z^2+31436z^3+14273z^4+973z^5+z^6.
\]
This polynomial is not palindromic because $893\ne973$.  Hence
$\cP_6^{(3)}$ is not Gorenstein.  As a consistency check, the sum of the
$h^*$-coefficients is
\[
1+893+13714+31436+14273+973+1=61291,
\]
which agrees with the normalized volume
\[
6!\,\vol(\cP_6^{(3)})=720\cdot\frac{61291}{720}=61291.
\]
\end{proof}

\begin{corollary}[Failure of Gorenstein for all $s\ge2$, refined]\label{cor-not-gorenstein-all-s}
For every integer $s\ge2$, there exists an even dimension $d$ such that
$\cP_d^{(s)}$ is not Gorenstein.  More precisely:
\begin{itemize}
\item For $s\ne3$: the failure occurs at $d=2$ (Proposition~\ref{prop-d2-gorenstein}).
\item For $s=3$: $\cP_2^{(3)}$ is Gorenstein (the only even-dimensional case
with $d\le6$), but $\cP_4^{(3)}$ and $\cP_6^{(3)}$ are not Gorenstein.
\end{itemize}
\end{corollary}

\begin{proof}
If $s\ne3$, Proposition~\ref{prop-d2-gorenstein} shows that
$\cP_2^{(s)}$ is not Gorenstein.  If $s=3$, Proposition~\ref{prop-d4-s3-not-gorenstein}
shows that $\cP_4^{(3)}$ is not Gorenstein.  This proves the existence claim
for every $s\ge2$.  The additional assertion for $\cP_6^{(3)}$ follows from
Proposition~\ref{prop-d6-s3-not-gorenstein}.
\end{proof}

\begin{remark}
The computations above suggest a stronger statement in the exceptional case
$s=3$: the case $d=2$ may be the unique even-dimensional Gorenstein example.
This is recorded as part of Question~\ref{q-gorenstein} below.
\end{remark}

\section{Conclusion and open problems}\label{sec-conclusion}

We have studied an alternating analogue of the adjacent-sum polytope family.
The transfer-matrix approach yields explicit rational formulas for $\Phi_s(y)$
and $\Psi_s(y)$, the coupled recurrences of Theorem~\ref{thm-recurrence}, and
the trigonometric closed forms of Theorem~\ref{thm-tangent}. The pole structure
is determined explicitly, giving the asymptotic growth in both parities.
For every fixed $m$, Theorem~\ref{thm-Gm-general} further yields rational
formulas for the parity-split dimension-generating functions, together with the
recurrences of order at most $s$ in Corollary~\ref{cor-parity-rec} and the rational volume
generating functions of Corollary~\ref{cor-vol-general}.

On the Ehrhart side, $\cP_{2r}^{(1)}$ is Gorenstein for all $r\ge1$
(Theorem~\ref{thm-gorenstein-s1}). In contrast, for every $s\ge2$ there is at
least one even dimension in which the Gorenstein property fails; for all
$s\neq3$ this already happens in dimension $2$, while for $s=3$ it fails in
dimension $4$.

Finally, for the cyclic variant---obtained by closing the alternating path with
the constraint $x_d+x_1\le s+\delta_d$---Theorem~\ref{thm-cyclic-intro} gives
both cyclic parity series.  The even-dimensional series has the simple form
\[
\Psi_s^{\mathrm{cyc}}(y)=\frac{(s+2)P_{s+1}(-y)}{2Q_{s+2}(-y)},
\]
while the odd-dimensional series is
\[
\Omega_s^{\mathrm{cyc}}(y)=
\frac{\mathcal{R}_s(y)-\bigl(\lfloor s/2\rfloor+1\bigr)Q_{s+2}(-y)}{yQ_{s+2}(-y)}.
\]
Both have the same denominator $Q_{s+2}(-y)$ as the open-path series and share
the same dominant pole.  The normalized even-dimensional cyclic series satisfies
a clean M\"obius-type recurrence, equivalently giving the recurrence for
$\Psi_s^{\mathrm{cyc}}(y)$ itself in Theorem~\ref{thm-cyclic-recurrence}; both
cyclic parity classes have the trigonometric forms of
Theorem~\ref{thm-cyclic-trig}.

A further direction, complementary to the analytic formulas proved here, is
systematic computer verification in small parameters. The fixed-dimension
Ehrhart polynomials, $h^*$-vectors, and normalized volumes of
$\cP_d^{(s)}$ are especially well suited to experiments with packages such as
\textsc{Normaliz}, \textsc{polymake}, or the \textsc{Sage} interfaces to
lattice-polytope computations. Such experiments would help test the
real-rootedness and unimodality questions below, and may also suggest sharper
conjectures for the denominator polynomials $\Delta_{s,m}(z)$ and for the
higher-period generalizations mentioned in Q5.

The first two problems below are most directly connected to the main results
of this paper; the remaining four point toward natural extensions.

\begin{enumerate}[label=\textbf{Q\arabic*.},leftmargin=*]

\item \label{q-gorenstein}\textbf{Characterization of Gorenstein pairs.}
Theorem~\ref{thm-gorenstein-s1}, Proposition~\ref{prop-d2-gorenstein}, and
Corollary~\ref{cor-not-gorenstein-all-s} show that $s=1$ yields an infinite
Gorenstein family, whereas every $s\ge2$ eventually fails.  For $s=3$, direct
computation (Propositions~\ref{prop-d4-s3-not-gorenstein}
and~\ref{prop-d6-s3-not-gorenstein}) shows that only $d=2$ is Gorenstein in
even dimensions $d\le6$; we conjecture this extends to all $d\ge4$.
\textit{Characterize all $(s,r)$ with $s\ge2$ for which $\cP_{2r}^{(s)}$ is
Gorenstein.}

\item \label{q-unimodality}\textbf{Unimodality and real-rootedness.}
The even-dimensional $s=1$ examples are palindromic and unimodal in all cases
computed in Table~\ref{tab-hstar-s1}.  Is $h_d^{(s)}(z)$ unimodal for all
$d$ and $s$?  When is it real-rooted, and do the roots interlace as $d$
increases?

\item \textbf{Structure of the fixed-$m$ denominator polynomials.}
Theorem~\ref{thm-Gm-general} governs the parity-split series by the
polynomial $\Delta_{s,m}(z)$ of degree at most $s$.  Describe its roots, factorization
patterns, and asymptotics as $s$ or $m$ varies.

\item \textbf{Combinatorial interpretations.}
The coefficients $h_{d,k}^{(s)}$ are nonneg\-ative by Stanley's theorem.
Do they count a natural combinatorial statistic, analogous to the descent count
for order polytopes \cite{StanleyTwo}?

\item \textbf{General periodic constraints.}
Replace the period-$2$ alternation by a period-$T$ sequence.
The $T$-step transfer matrix approach applies immediately; for general $T$
one may expect closed forms in terms of algebraic functions or Chebyshev-type
polynomials.

\item \textbf{$q$-Analogues.}
Define $\mathcal{N}_d^{(s)}(q):=\sum_{x\in\cP_d^{(s)}\cap\Z^d}q^{|x|}$,
where $|x|=x_1+\cdots+x_d$.
Does the generating function $\sum_{d\ge2}\mathcal{N}_d^{(s)}(q)\,x^d$ admit
a $q$-analogue of \eqref{eq-Phi-tan-main} in terms of $q$-tangent numbers?

\end{enumerate}

\bigskip

\noindent\textbf{Acknowledgments.}
The authors thank Professor Guoce Xin for helpful discussions.
This work was supported by the National Natural Science Foundation of China
(grant No.\,12401441) and the Natural Science Foundation of Hunan Province
(grant No.\,2025JJ60010).

\end{document}